\documentclass[reqno]{amsart}
\usepackage[dvipsnames]{xcolor}

\usepackage[normalem]{ulem}
\usepackage{amsfonts,amsmath,amssymb,mathrsfs,amsthm,url,esint,comment
}
\usepackage{babel}
\usepackage[normalem]{ulem}
\usepackage{hyperref}
\usepackage{cleveref}
\usepackage{enumitem}
\usepackage{color}

\usepackage{graphicx}
\usepackage{xcolor}
\usepackage{epsfig}
\usepackage{float}
\usepackage{pdfsync}
\usepackage{wasysym}
\usepackage{empheq}
\usepackage{textgreek}
\usepackage{url}
\usepackage[left=2cm,right=2cm,top=2cm,bottom=2.5cm]{geometry}
\definecolor{green1}{rgb}{0.0, 0.5, 0.0}

\graphicspath{{./figure/}}

\newtheorem{definition}{Definition}[section]

\newtheorem{thm}[definition]{Theorem}
\newtheorem{theorem}[definition]{Theorem}
\newtheorem{lemma}[definition]{Lemma}
\newtheorem{prop}[definition]{Proposition}
\newtheorem{cor}[definition]{Corollary}

\newtheorem{defn}[definition]{Definition}

\theoremstyle{remark}
\newtheorem{example}[definition]{Example}
 \newtheorem{rem}[definition]{Remark}

\newcommand{\B}{\mathbb{B}}
\newcommand{\R}{\mathbb{R}}
\newcommand{\N}{\mathbb{N}}
\newcommand{\T}{\mathbb{T}}
\newcommand{\G}{{\Gamma}}

\newcommand{\cX}{\mathcal{X}}
\newcommand{\cB}{\mathcal{B}}
\newcommand{\cC}{\mathcal{C}}
\newcommand{\cJ}{\mathcal{J}}
\newcommand{\cD}{\mathcal{D}}

\newcommand{\cF}{\mathcal{F}}

\newcommand{\cH}{\mathcal{H}}
\newcommand{\cI}{\mathcal{I}}

\newcommand{\cL}{\mathcal{L}}

\newcommand{\cS}{\mathcal{S}}
\newcommand{\cU}{\mathcal{U}}
\newcommand{\cY}{\mathcal{Y}}

\newcommand{\wt}{\widetilde}

\newcommand{\koo}{{k\to\infty}}

\newcommand{\wto}{\rightharpoonup}
\newcommand{\weaks}{\stackrel{\star}{\wto}}

\newcommand{\cA}{\mathcal{A}}

\newcommand{\hp}{p}

\newcommand{\hx}{x}
\newcommand{\hu}{u}

\newcommand{\loc}{{\rm loc}}

\newcommand{\dis}{\displaystyle}
\newcommand{\Wl}{W_{\rm loc}}

\newcommand{\seq}{{\rm seq}}

\DeclarePairedDelimiterX{\inp}[2]{\langle}{\rangle}{#1, #2}

\newcommand{\AI}{{\cA}}

\newcounter{altassumption}[assumption]
\renewcommand{\thealtassumption}{\theassumption w}

\author{Matteo Della Rossa}
\author{Lorenzo Freddi} 
\address[ Matteo Della Rossa, Lorenzo Freddi]{Dipartimento di Scienze Matematiche, Informatiche e Fisiche,  Universit\`a di Udine, via delle Scienze 206, 33100 Udine, Italy}
\email{matteo.dellarossa@uniud.it, lorenzo.freddi@uniud.it}
\title[Pattern-preserving optimal control problems]{Pattern-preserving optimal control problems \\ with increasing time horizon}

\begin{document}

\maketitle

\begin{abstract}
 We establish a general  framework that guarantees the preservation of optimal control patterns as the time horizon $[0,T]$ increases and becomes unbounded.
A concept of pattern-preserving family of optimal control problems is introduced and 
the goal is achieved by analyzing the $\Gamma$-convergence of the corresponding variational formulations  as $T\to\infty$.  Special attention  is given to scenarios involving state constraints. To illustrate the results, examples and applications are provided, with particular focus on switched systems and epidemic control.
\end{abstract}

\maketitle

\medskip


\section{Introduction}
In this paper, we aim to establish some sufficient conditions under which the structure of the optimal controls of finite-horizon problems is 
inherited by at least one solution of the  corresponding problem on the infinite horizon. Such problems naturally arise, for instance, in epidemic control when it is desirable to conduct the study without constraining the epidemic horizon to a bounded interval (see, e.g, \cite{FreddiGoreac23,DellaRossaFreddi24}). Similar applications can be imagined also in other branches of optimal control theory such as  switching control systems, \cite{Lib03,ChitMasSig25}. Our approach is to consider a  variational formulation of the control problems and study the variational convergence of such sequence (that is, the convergence of optimal controls with final time $T$ to an optimal control of the infinite horizon problem)
by the celebrated tool of De Giorgi's $\G$-convergence (see~\cite{Braides02,DalMaso93} for an overview).    
The optimal control problems are thus formulated as a variational problem
\[
\inf_{\cU\times \cX}\cF_T=\inf_{ \cU\times \cX}\big(\cJ_T+\chi_\cA\big),
\]
where $\cU$ is the \emph{control space}, $\cX$ is the \emph{state space}, $\cF_T: \cU\times \cX\to (-\infty,+\infty]$ is the \emph{joint functional} defined  as the sum of the \emph{cost functional} $\cJ_T:\cU\times \cX\to (-\infty,+\infty]$ and the indicator function $\chi_\cA$ taking the value zero on the set $\cA\subset  \cU\times \cX$ of \emph{admissible pairs} and $+\infty$ otherwise. The parameter $T\in (0,\infty]$ in our setting represents the \emph{final time}, while $[0,T]$ is the corresponding \emph{horizon}.
We are then led to study the $\Gamma$-convergence of the functionals $\cF_T$ as $T\to\infty$.

Variational convergence of optimal control problems is an established and fruitful path of research. In the seminal paper~\cite{ButtDalMas82}, the authors employed  $\Gamma$-convergence tools to provide conditions under which optimal pairs for a sequence of parametrized problems converge to an optimal pair of a limiting problem. Then, several extensions and related results have been provided in the literature; see for example~\cite{Buttazzo89,BC89, F2000, BelButFre93} and references therein. 

A novelty of this manuscript lies in the fact that the parameter $T$, in the considered sequences of optimal control problems, represents the \emph{final time} which, to the best of our knowledge, is a case not considered before in the literature in a general framework. 
This setting, and in particular the ``moving'' final time, introduces new technical challenges that cannot be overcome by existing results. One of them is that it is required to work on spaces of functions defined on unbounded domains (the semi-line $[0,\infty)$) and to introduce tailored \emph{local} Sobolev spaces. 

Another feature of our setting is that it allows to prescribe  \emph{state-constraints} of the form $x(t)\in X$ for almost all $t\in [0,\infty)$, for a given set $X$. 

Unfortunately, these peculiarities have rendered unsuccessful any attempt to use the aforementioned early results on $\G$-convergence of optimal control problems, because the latter are based on a continuity property of the cost functional with respect to the state variable that seems very hard to satisfy in our setting.
For these reasons, the $\G$-convergence is established by using novel proof techniques, in order to handle the particular  features of our framework.

Once  $\G$-convergence has been  established for a family of functionals $\cF_T$ (representing optimal control problems on the finite horizon $[0,T]$), under a suitable \emph{coercivity} assumption we are able to prove  pattern preservation  for such family: if the optimal controls for $\cF_T$ share a common structure as $T$ diverges, the same structure is also inherited by an infinite-horizon optimal control. This is particularly relevant since, in many cases,  optimality conditions are more easily derived for finite-horizon problems, for instance, via Pontryagin's principle (see for example~\cite[Chapter 9]{Vinter2010Book}). Indeed, although Pontryagin-like techniques exist to  address infinite-horizon control problems (see for example~\cite{Halkin1974,AseeKry04,CANNFra18,BasCassFra18} and references therein), they often yield weaker or less tractable necessary conditions compared to their finite-horizon counterparts. This is primarily due to the lack of a ``natural'' terminal transversality condition at the final time 
$T$ (which in these cases is equal to $\infty$).

The underlying idea of our work is inspired by recent papers in which some specific problems arising from epidemic control
are studied (\cite{FreddiGoreac23,DellaRossaFreddi24}). In the latter, the $\Gamma$-convergence tool is used to solve infinite-horizon optimal controls with state-equations defined by a (rather structured) bi-dimensional controlled SIR model (see~\cite{anderson1992infectious}  for an overview of such epidemic models).
Our manuscript extends these case-specific results into an abstract and more comprehensive framework that encompasses several other  models studied in optimal control theory. Indeed, the considered  admissible set $\cA$ will be defined by general \emph{controlled differential equations} of the form
\begin{equation}\label{eq:DiffInclINTRO}
x'(t)= a(t,x(t))+b(t,x(t))u(t)
\end{equation}
where $a,b$ are  Carath\'eodory functions and $u$ is the control variable, while the cost functionals are of the  form
\[
\cJ_T(u,x)=\int_0^T \ell(t,x,u)+\chi_X(x)\,dt
\]
where $\ell$ is a normal convex integrand and the indicator  $\chi_X(x)$ is added to model the  state-constraint.

Among the final examples, besides  the aforementioned epidemic case, we give particular attention to switched systems, for which original results are obtained regarding the structure of optimal controls over an infinite horizon.

The paper is organized as follows.  In Section~\ref{sec:Preliminaries} we  give the definition of pattern-preserving sequences of optimal control problems and establish a general theorem relating it with $\Gamma$-convergence theory. In Section~\ref{sec:MainResults} we collect, without proofs, the main results of the paper concerning  more specific sequences of optimal control problems. In Section~\ref{sec:ODEandchar} and~\ref{Sec:Boundedness}, we introduce and study the basic properties of   controlled differential equations defining the set of admissible pairs. In Section~\ref{sec:ParticularCases}, we study some issues concerning coercivity for two remarkable classes of control systems falling in our theory.   
 Section~\ref{Sec:StateConstr} is devoted to prove the  pattern preservation results presented in Section \ref{sec:MainResults}. 
 Our  results are illustrated in Section~\ref{sec:Examples} through examples drawn from switched systems and epidemic control. 
In Appendix~\ref{subsec_locSob} we collect some preliminaries concerning local Sobolev spaces and their topologies. This functional setting is the natural one for our space of states.  Although it does not contain particularly original results, we have chosen to retain this section since, to the best of our knowledge, the preliminary material it collects is not presented in an organized manner elsewhere in the literature.
 \
 
\textbf{Notation:} $\overline \R:=[-\infty,\infty]$ denotes  the set of extended real numbers.
By $|\cdot|$ we denote any norm in $\R^k$, $k\in\N$; when applied to matrices, it denotes the associated \emph{operator norm}. By $\B(x_0,r)$ we denote the open ball wit center $x_0$ and radius $r$ in the considered norm. The usual norm in the Lebesgue space $L^p$ will be denoted by $\|\cdot\|_p$. Given $p\in[1,\infty]$, as usual, $p'$ denotes the conjugate exponent of $p$, defined by the equality $\frac{1}{p'}+\frac{1}{p}=1$ and the convention $1/\infty=0$.  
By $\cD(\Omega,\R^k)$ we denote the space of vector valued test functions on the open set $\Omega$ and the dual  $\cD'(\Omega,\R^k)$ is the space of distributions.   Given a set $A$,  we define 
\[
\chi_A(x):=\begin{cases}
0&\text{if }x\in A\\
+\infty&\text{if }x\notin A
\end{cases}\ \mbox{ and }\ \mathbf{1}_A(x):=\begin{cases}
1&\text{if }x\in A\\
0&\text{if }x\notin A
\end{cases}
\]
the \emph{indicator} and \emph{characteristic}  functions of $A$, respectively.
We will also make use of the following standard definition (see, e.g., \cite[Definition 2.1.1]{Buttazzo89}):
a function $\ell:[0,\infty)\times \R^n\times \R^m\to (-\infty,+\infty]$  is said to be 
\begin{enumerate}[leftmargin=*]
\item an \emph{integrand} if it is measurable with respect to the Lebesgue-Borel measure $\cL\otimes \cB_n\otimes \cB_m$ on $[0,\infty)\times \R^n\times \R^m$;
\item a \emph{normal integrand}  if it is an integrand and $\ell(t,\cdot,\cdot)$ is lower semicontinuous for almost all $t\ge0$;
\item a \emph{normal convex integrand} if it is a normal integrand and, for almost all $t\ge0$, the map  $u\mapsto \ell(t,x,u)$ is convex for all $x\in \R^n$.
\end{enumerate}

\section{$\Gamma$-convergence and pattern-preserving sequences }\label{sec:Preliminaries}
 For a general introduction to $\Gamma$-convergence and related topics, we refer to the monographs~\cite{DalMaso93,Braides02}.\\

Let us consider two topological spaces $\cU$ and $\cX$, named the space of \emph{controls} and \emph{states}, respectively. 
Given a \emph{cost functional} $\cJ:\cU\times \cX\to \overline \R$ and a set of \emph{admissible control-state pairs} $\cA\subset \cU\times \cX$, we consider the \emph{optimal control problem} defined by
\begin{equation}\label{eq:AbstractOptContrProb0}
\begin{array}{c}
\displaystyle\inf_{\cU\times \cX}\cJ(u,x)\\[-1ex]
\ \\[-1ex]
\text{subject to}\\[-1ex]
\ \\[-1ex]
(u,x)\in \cA.
\end{array}
\end{equation}
By considering the indicator function of $\cA$, problem~\eqref{eq:AbstractOptContrProb0} takes the form 
\begin{equation}\label{eq:AbstractOptContrProb}
\inf_{\cU\times \cX} (\cJ+\chi_{\cA}),
\end{equation}
that is, the optimal control problem is rewritten as a minimization problem  for a {\em joint functional}  $\cF:=\cJ+\chi_\cA$. With a small abuse of language, we will speak of {\em the optimal control problem for $\cF$}.  
As usual,  we say that $u\in \cU$ is an \emph{optimal control} for $\cF$ if there exists $x\in \cX$ such that $\cF(u,x)=\inf_{\cU\times\cX}\cF$.

The topologies on the spaces $\cU$ and $\cX$ start to play a role when, besides the single problem, we 
consider a one-parameter family of minimization problems for functionals $\cF_T:=\cJ_T+\chi_{\cA_T}$, where the parameter $T\ge0$ plays the role of time, and look for a variational limit as $T\to\infty$, with respect to the convergences induced by the chosen topologies.

An  appropriate notion of variational convergence  is provided by  sequential $\Gamma$-limits, whose definitions for a general sequence $\cF_k:\cU\times \cX\to \overline \R$ are
$$
\begin{array}{l}
\displaystyle\G^-_\seq(\cU\times\cX)\liminf_\koo \cF_k(u,x):=\inf_{u_k\to u}\inf_{x_k\to x}\liminf_\koo
\cF_k(u_k,x_k),\\
\displaystyle\G^-_\seq(\cU\times\cX)\limsup_\koo \cF_k(u,x):=\inf_{u_k\to u}\inf_{x_k\to x}\limsup_\koo
\cF_k(u_k,x_k).
\end{array}
$$  When they coincide, we denote their common value by writing
$$
\G_\seq^-(\cU\times\cX)\lim_\koo \cF_k.
$$

 \begin{rem}\label{rem_lirs}
It is easy to check that a sufficient conditions implying $\G_\seq^-(\cU\times\cX)$-convergence of $\cF_k$ to $\cF$ is given by 
\begin{enumerate}
\item (\emph{liminf inequality}) for all sequences $(u_k,x_k)\to (u,x)$ we have $\dis \cF(u,x)\leq \liminf_{k\to \infty} \cF_k(u_k,x_k)$;
\item (\emph{recovery sequence}) there exists a sequence $(u_k,x_k)\to (u,x)$ such that  $\dis \cF(u,x)= \lim_{k\to \infty} \cF_k(u_k,x_k)$.
\end{enumerate}
Clearly, 1.\ is also necessary.  On the other hand, the necessity of 2., that is the equivalence (between 1.-2.\ and $\G_\seq^-$-convergence) 
holds in first countable spaces (as well as the equivalence between topological and sequential $\G$-limits, see \cite[Proposition 8.1]{DalMaso93}) or for equi-coercive sequences in spaces in which relatively compact sets are metrizable like, e.g., Banach spaces with a separable dual space, or dual of a separable Banach space, with the weak, respectively weak*, topology (see  \cite{AS1976} or \cite{DalMaso93}).  
This is, in fact, our case. Nevertheless, in the sequel, only the sufficiency  will be used (in the proof of Theorem \ref{thm:MainTheorem}).
\end{rem}

Even though they do not always coincide with the topological ones, sequential $\Gamma$-limits retain many good properties useful in applications. In particular, also sequential $\Gamma$-convergence has a {\em variational character}, that is, it ensures convergence of minima 
and minimizers (under equi-coercivity) according to the following definitions and  theorem (\cite[Proposition 2.1]{ButtDalMas82}; see also \cite{Buttazzo89,Braides02}). 

\begin{definition}\label{defn:EquiCoercivityGeneral}
We say that  $(\cF_k)$ is {\em sequentially equi-coercive} if, for every $C>0$ and  every sequence $(u_k,x_k)$ such that $\cF_k(u_k,x_k)\le C$, there exists a subsequence of $(u_k,x_k)$ converging to some  $(u,x)$ in $\cU\times\cX$.
\end{definition}

\begin{definition}\label{def_ms}
    A sequence $(u_k,x_k)$ 
in  $\cU\times \cX$ is said to be {\em minimizing} for the sequence $\cF_k$ if 
\begin{equation}\label{lili}
\liminf_{k\to\infty}\cF_{k}(u_k,x_k)=\liminf_{k\to\infty}\inf_{\cU\times \cX}\cF_{k}.
\end{equation} 
\end{definition}

\noindent It is easy to see that a minimizing sequence always exists.
The sequences of optimal pairs (if the optima exist) are trivially minimizing.  
Moreover, in the case of a constant sequence  (i.e.,  $\cF_k=\cF$ for every $k$) the previous Definition \ref{def_ms} 
reduces to the classical one for the functional $\cF$.

\begin{thm}[variational property]\label{lemma:GammaCOnvMAINProp}
Consider a sequence $\cF_k:\cU\times \cX\to \overline \R$  that  $\Gamma^-_{\rm seq}(\cU\times\cX)$-converges to $\cF$ and
let $(u_k,x_k)$ be a minimizing sequence 
The following propositions hold.
\begin{enumerate}
\item     
If $(u_k,x_k)\to (u,x)\in \cU\times \cX$, then 
\begin{equation}\label{eq_F=infF}
\cF(u,x)=\inf_{\cU\times \cX}\cF=\lim_\koo\big[\inf_{\cU\times \cX}\cF_{k}\big];
\end{equation}
actually, $(u,x)$ is an optimal pair for $\cF$.
\item   If $\cF\not\equiv+\infty$ and  $(\cF_{k})$ is sequentially equi-coercive,  then 
\begin{enumerate}
\item there exists an optimal pair $(u,x)$ of $\cF$ and a subsequence 
$(u_{k_n},x_{k_n})\to (u,x)$ in $\cU\times \cX$;
\item if, moreover,   
\begin{enumerate}
\item $\cX$ is metrizable
\item for every optimal control $u\in\cU$ there exists at most one $x\in\cX$ such that $\cF(u,x)<\infty$,
\end{enumerate}
then
\begin{center}
 $u_k\to u\in \cU$ implies  $x_k\to x$ and the equalities \eqref{eq_F=infF} hold;
 \end{center}
actually, $u$ is an optimal control for $\cF$ and $x$ is the (unique) corresponding optimal state.
\end{enumerate}
\end{enumerate}
\end{thm}



\begin{proof}
Part {\em1.}\ of the statement has been proven in \cite[Proposition 2.1]{ButtDalMas82}. To prove part {\em2.}, we start by observing that $\displaystyle\liminf_{k\to\infty}\inf_{\cU\times \cX}\cF_k$ must be smaller than $+\infty$ since, otherwise, by definition of $\G_\seq$-liminf we immediately would have $\cF\equiv+\infty$, which is excluded. Hence, by \eqref{lili}, the sequence $\cF_k(u_k,x_k)$ is upper bounded and, by equi-coercivity, there exists a subsequence $(u_{k_n},x_{k_n})$ converging to some $(u,x)$ in $\cU\times \cX$. By part~{\em1.}\ of the theorem, $(u,x)$ turns out to be a minimum point of $\cF$ and the proof of this part {\em2.(a)} is concluded. 

It remains to prove {\em 2.(b)}.   Let $u_k\to u$ in $\cU$. As before,  there exists a subsequence $x_{k_n}$ converging to some $x$ in $\cX$. By the previous part of the theorem, $(u,x)$ turns out to be an optimal pair of $\cF$ and the equalities \eqref{eq_F=infF} hold. The same argument applies to any subsequence of $x_k$ that, therefore, admits a subsequence converging always to the same $x$ (by hypothesis {\em ii.}). Since $\cX$ is metrizable, than the whole sequence converges to $x$ and 
the theorem is completely proved. 
\end{proof}

\begin{rem}
Part {\em 2.(b)} of the theorem is less standard but applies, in particular, to optimal control problems like that  of  this paper in which the state equation admits at most one solution for every choice of the control,  and the states belong  to a metric  space $\cX$. 
\end{rem}

The variational property of $\Gamma$-convergence can be applied to show that, 
under $\G$-convergence and coercivity assumptions, if the \emph{optimal control policies} on the  finite horizons share a
certain common structure, the same control pattern will apply also to the infinite-horizon problem.

This result is obtained for a specific choice of the control space that, from this point on and for the rest of the paper, will be   
\begin{equation}\label{defU}
\cU=L^p((0,\infty),U), \mbox{ with $p\in (1,\infty]$,}
\end{equation}
and 
$$
U=\begin{cases}
    \R^m &\mbox{if } p<\infty\\
     \mbox{a closed and convex subset of } \R^m&\mbox{if } p=\infty,
\end{cases}
$$
equipped with its \emph{weak$^\star$}, or $\sigma (L^{p},L^{p'})$, topology coming from the duality $L^p=(L^{p'})^*$. We note that it is appropriate to speak of weak$^\star$ topology also for $p$  finite, remembering that, in such reflexive cases, it coincides with the {\em weak} one. Since the topology has been fixed once and for all,  the convergence in this space will be denoted by the generic symbol $\to$. For simplicity, we use this notation and terminology along the whole paper.

\begin{defn}[pattern-preserving family]\label{defn:PatternPreserving}
 Let be given a family 
$\cF_T:\cU\times \cX\to (-\infty,+\infty]$, $T\in(0,\infty]$, of optimal control problems and a sequence $0<T_k\to\infty$ as $k\to\infty$. The family $(\cF_T)$   
is said to be  $(T_k)$-{\em pattern preserving} if the following property is satisfied: 
 for any minimizing sequence  $(u_k,x_k)\in \cU\times \cX$ 
of $(\cF_{T_k})$ with 
\begin{enumerate}
\item     ${u_k}_{\vert_{[0,T_k]}}$ represented   in the form
\begin{equation}\label{eq:OptimalControlsPiecewise}
{u_k}_{\vert_{[0,T_k]}}=\sum_{j=1}^N u_j^k\mathbf{1}_{[\tau^k_{j-1},\tau^k_{j})}
\end{equation}
for suitable  $N\in \N\setminus\{0\}$, $u^k_1,\dots, u^k_N\in \cU$ and a partition $0=\tau^k_0\leq \tau_1^k\leq \dots\leq \tau_{N-1}^k\leq \tau^k_N=T_k$, 
\item and such that 
\begin{equation}\label{tjkujk}
\begin{array}{l}
\tau^{k}_j\to \tau^{\infty}_j\mbox{ in }\overline{\R},\\[1ex]
u_j^k\to u_j^\infty\mbox{ in }\cU,
\end{array}
\mbox{ for any }
j\in \{1,\dots, N\},
\end{equation} 
\end{enumerate}
it turns out that
\[
u_\infty:=\sum_{j=1}^{N} u_j^\infty\mathbf{1}_{[\tau^\infty_{j-1},\tau^\infty_{j})}\quad 
\mbox{ (with $[+\infty,+\infty):=\varnothing$)}
\]
is an optimal control for $\cF_\infty$.
\end{defn}

The definition may appear cumbersome at first glance. In particular, it says  that if the $\cF_{T_k}$ have optimal controls $u_k$ 
with a    \eqref{eq:OptimalControlsPiecewise}-like piecewise structure, then the limiting control problem 
$\cF_\infty$
  admits an optimum with the same structure, whose coefficients and partition points are the limits of those of $u_k$.

As a first result, we show that (sequential) $\Gamma$-convergence and  equi-coercivity are sufficient conditions for ensuring pattern preservation.

\begin{thm}[pattern preservation]\label{prop:COnvergenceofOptima}  
Consider a parametrized family of optimal control problems $\cF_T:\cU\times \cX\to (-\infty,+\infty]$ and suppose that there exists a sequence ${\color{red}0<}T_k\to \infty$ such that
\begin{enumerate}
    \item $\cF_{T_k}$ is sequentially equi-coercive;
    \item $\displaystyle \G^-_{\rm seq}(\cU\times\cX)\lim_{k\to\infty}\cF_{T_k}=\cF_\infty$.
\end{enumerate} 
 Then $\cF_T$ is $(T_k)$-pattern preserving.
\end{thm}

\begin{proof}
We can suppose that $\cF_\infty \not\equiv +\infty$, otherwise the claim is trivially satisfied.
Consider a minimizing  
 sequence  $(u_k,x_k)\in \cU\times \cX$
 with ${u_k}_{\vert_{[0,T_k]}}$  in the form~\eqref{eq:OptimalControlsPiecewise} and satisfying~\eqref{tjkujk}
for a suitable partition  $0=\tau^k_0\leq \tau_1^k\leq \dots\leq \tau_{N-1}^k\leq \tau^k_N=T_k$.
Note that such a sequence always exists. 

Under our hypotheses, the variational property of $\G$-convergence (Theorem~\ref{lemma:GammaCOnvMAINProp}, Point~\emph{2.}) implies that $(u_k,x_k)\to (u,x)$ in $\cU\times \cX$ (up to a subsequence), where $(u,x)\in \cU\times \cX$ is such that $\cF_\infty(u,x)=\inf_{\cU\times \cX}\cF_\infty$, i.e., is an optimal pair for the limiting problem $\cF_\infty$.

The thesis follows by showing that $\lim_{k\to \infty}u_k=u_\infty$ in the sense of distributions and by uniqueness of the limit.
For any $\varphi\in \cD((0,\infty),\R^m)$  we have that 
\[
\begin{aligned}
\lim_{k\to \infty}\int_0^\infty u_{k}(t)\cdot \varphi(t)\,dt&
=\sum_{j=1}^{N}\lim_{k\to \infty}
\int_{0}^{\infty} u_j^k(t)\cdot\varphi(t)\mathbf{1}_{[\tau^k_{j-1},\tau^k_{j}]}(t)\,dt
\\&=\sum_{j=1}^{N}
\int_{0}^{\infty} u_j^\infty(t)\cdot\varphi(t)\mathbf{1}_{[\tau^\infty_{j-1},\tau^\infty_{j}]}(t)\,dt=\int_0^\infty u_{\infty}(t)\cdot\varphi(t)\,dt
\end{aligned}
\]
where the last but one equality follows by the fact that 
\begin{equation}\label{eq:StrongConvergenceVarPhi}
\varphi\mathbf{1}_{[\tau^k_{j-1},\tau^k_{j}]}\stackrel{k\to\infty}{\longrightarrow}
\varphi\mathbf{1}_{[\tau^\infty_{j-1},\tau^\infty_{j}]}\quad\text{ strongly in } L^{p'}((0,\infty),\R^m)
\end{equation}
while 
$u_j^k\to u_j^\infty$ weakly* in $L^p((0,\infty),\R^m)$.
The strong convergence in~\eqref{eq:StrongConvergenceVarPhi} follows by observing  that
$$
\|\varphi\mathbf{1}_{[\tau^k_{j-1},\tau^k_{j}]}- 
\varphi\mathbf{1}_{[\tau^\infty_{j-1},\tau^\infty_{j}]}\|_{p'}^{p'}=\int_{E_k}|\varphi(t)|^{p'}\,dt
$$
where 
$$
E_k=\big([\tau^k_{j-1},\tau^k_{j}]\Delta [\tau^\infty_{j-1},\tau^\infty_{j}]\big)\cap\text{supp}\,\varphi
$$
being $\Delta$ the symmetric difference between sets,  
and using the fact that, since the support of $\varphi$ is compact, the Lebesgue measure of the set $E_k$ tends to $0$  as $k$ goes to $\infty$.
 This proves that $u_{k}\to u_\infty$ in $\cU$, and the proof is concluded.
\end{proof}


\begin{rem}[$p=\infty$, limit of piecewise constant controls]\label{rem:BangsBangs}
Let $p=\infty$ and, according to \eqref{defU}, $\cU=L^\infty((0,\infty),U)$ with the weak$ ^\star$ topology. 
Given $T_k\to \infty$, suppose that $(u_k,x_k)\in \cU\times \cX$ be  a minimizing sequence such that the restrictions  ${u_k}_{\vert_{[0,T_k]}}$ be \emph{piecewise constant} with the same number of discontinuities and the same values, i.e., of the form~\eqref{eq:OptimalControlsPiecewise} with  $u_j^k\equiv u_j\in \R^m$ for all $k\in \N$ and for all $j\in \{1,\dots, N\}$. Under the hypothesis of Theorem~\ref{prop:COnvergenceofOptima} suppose that  $\tau^{k}_j\to \tau^{\infty}_j\in \overline \R$,  for every $j$. 
 Then, there exists  $u_\infty \in \cU$ optimal for $\cF_\infty$ with at most $N$ discontinuity points (and thus eventually constant) of the form 
\[
u_\infty:=\sum_{j=1}^{N} u_j\mathbf{1}_{[\tau^\infty_{j-1},\tau^\infty_{j})}.
\]
Actually, this  formalizes the intuitive idea that, if it can be established (for instance, via  Pontryagin's  principle) that the finite-horizon optimal controls exhibit a \emph{bang-bang} structure then, under equi-coercivity, 
the same structure will also be inherited by (at least one) optimal control of the $\Gamma$-limit \emph{infinite-horizon}  problem. 
A  further discussion  on such property and some examples will be provided in following sections.
\end{rem}

\begin{rem}\label{rem_GammaT}
We conclude this section by observing that  sometimes, throughout the paper, it will be more effective to speak about $\Gamma$-convergence of a family $\cF_T$ of functionals indexed by a parameter $T\in(0,+\infty)$, towards a limit functional $\cF_\infty$. By this we  mean that $\cF_{T_k}$ $\Gamma$-converges to $\cF_\infty$
for every sequence $0< T_k\to\infty$. 
\end{rem}

\section{Running assumptions and main results}\label{sec:MainResults}

In the previous section  we have introduced the definition of pattern-preserving sequence of optimal control problems  and  
 established our first result  (Theorem  \ref{prop:COnvergenceofOptima})
for a totally general sequence of problems. In this section  we collect the other main results of the paper concerning  more specific sequences of optimal control problems,  postponing   proofs and  discussion of the assumptions to the next sections.

Actually, we consider  a sequence of  problems 
parametrized by the final time $T\in (0,+\infty]$, with cost functionals of the form
\begin{subequations}\label{eq:OptimalControlProblem}
\begin{equation}\label{eq:CostFunctionalSec3}
J_T(u,x):=\int_0^T \ell(t,x(t),u(t))\,dt
\end{equation}
to be minimized over all admissible control-state  pairs $(u,x)$ belonging to the    
 set 
\begin{equation}\label{eq:FeasibleLinearized}
\begin{aligned}
\AI
 &:=\{(u,x)\in \cU\times\cX\; :\;  x=x^{u,x_0}\},
\end{aligned}
\end{equation}
 where $x^{u,x_0}\in \cX:=W^{1,1}_{\loc}([0,\infty),\R^n)$ denotes the unique (by the assumptions below) global  solution to the Cauchy problem 
\begin{equation}\label{eq:CauchyProblem1}
\hspace{2.6cm}\begin{cases}
x'(t)
=a(t,x(t))+b(t,x(t))u(t)\\
x(0)=x_0.
\end{cases}
\end{equation}
It is worth noting that, here and in the following, the locality in the involved Sobolev (and Lebesgue) spaces is meant in the sense of  restriction to any interval of the form $[0,T)$, $T>0$; for a formal definition and further details we refer to the next section and/or  to Appendix \ref{subsec_locSob}.

The running cost  $\ell$ is allowed to take extended real values and, therefore, can incorporate an additive term of the form $\chi_{X}(x)$ where   $X\subseteq\R^n$ is a closed set representing a \emph{state constraint}. In other terms we   model a state constraint 
of the form 
\begin{equation}\label{mr:sc}
x(t)\in X\   \forall t\in[0,T)
\end{equation}
by requiring that $\ell(\cdot,x,\cdot)\equiv+\infty$ whenever  $x\in\R^n\setminus X$.

In the case $p=\infty$, the choice of the space $\cU$ (see \eqref{defU})  prescribes a constraint on the controls of the type 
\begin{equation}\label{eq:COntrolCOntraint} 
u(t)\in U\ \mbox{for a.a.} \ t\in(0,\infty),
\end{equation}
\end{subequations}
for a convex and closed set $U\subset \R^m$.


Accordingly, the joint functional $F_T:\cU\times\cX\to (-\infty,+\infty]$,  that completely represents the optimal control problem, is defined by $F_T:=J_T+\chi_{\AI}$.

On the state-equations defining the set $\AI$ we assume
Carath\'edory-Lipschitz hypotheses (A0) and a global existence assumption (A1). 
Precisely,  
\begin{itemize}
\item[(A0)]  the functions $a:[0,\infty)\times \R^n\to \R^n$ and  $b:[0,\infty)\times \R^n\to \R^{n\times m}$  are Borel measurable and such that \begin{enumerate}
\item  there exists  $N\in L^{p'}_{\rm loc}([0,\infty))$ and, for any compact set $H\subset\R^n$, there exists $B_H\in L^{p'}_{\rm loc}([0,\infty))$ such that 
\[
\begin{aligned}
 &|b(t,0)|\leq N(t)\;\;\mbox{ for a.a.}\  t\ge0,\\
&|b(t,x_1)-b(t,x_2)|\leq B_H(t)|x_1-x_2|\quad\text{for all }  x_1, x_2 \in H\mbox{ and a.a.}\ t\ge0;
\end{aligned} 
\]
\item  there exists $M\in L^{1}_{\rm loc}([0,\infty))$ and, for any compact set $H\subset\R^n$, there exists $A_H\in L^{1}_{\rm loc}([0,\infty))$ such that
\[
\begin{aligned}
 &|a(t,0)|\leq M(t)\;\;\mbox{ for a.a.}\  t\ge0,\\
&|a(t,x_1)-a(t,x_2)|\leq A_H(t)|x_1-x_2|\quad\text{for all }  x_1, x_2 \in H\mbox{ and a.a.}\ t\ge0.
\end{aligned} 
\]
\end{enumerate} 
\item[(A1)]\label{ass:globalexistence1}
for any $u\in \cU$ and any sequence $u_k\to u$ in $\cU$, we have that: if $x^{u,x_0}\in W^{1,1}_{\loc}([0,\infty),\R^n)$ then 
$x^{u_k,x_0}\in W^{1,1}_{\loc}([0,\infty),\R^n)$ 
for any $k$ large enough.
\end{itemize}

On the running cost function $\ell:[0,\infty)\times \R^n\times \R^m \to [0,+\infty]$ in~\eqref{eq:CostFunctionalSec3} we make the following assumptions. Let us suppose that there exist a non-negative  normal integrand $\ell_1:[0,\infty)\times\R^n\to [0,+\infty]$ and a non-negative normal convex integrand
$\ell_2:[0,\infty)\times \R^n\times \R^m \to [0,+\infty]$ such that 
\begin{enumerate}[label=(\alph*)]\setcounter{enumi}{0}
\item $\ell(t,x,u)=\ell_1(t,x)+\ell_2(t,x,u)$ for all $(t,x,u)\in [0,\infty)\times \R^n\times \R^m$;
\item \emph{(existence of a  bounded greedy control strategy)} there exists $u_g:(0,\infty)\times \R^n\to \R^m$ Borel-measurable  such that $\ell_2(t,x,u_g(t,x))=0$ for almost all $t\ge0$ and all $x\in \R^n$ and
\begin{enumerate}
\item[(b1)] 
  if $p<\infty$, for every $\cY\subset \cX$ bounded in  $L^\infty_{\loc}([0,\infty),\R^n)$  there exists $\alpha_\cY\in L^p(0,\infty)$ such that
$$
|u_g(t,x(t))|\le\alpha_\cY(t)\quad \forall\, x\in \cY \mbox{ and a.a.}\ t\ge0;
$$
\item[(b2)] if $p=\infty$, we have 
$u_g(t,x)\in U$ $\forall\,x\in \R^n$ and a.a.\ $t\ge0$.
\end{enumerate}
\item {\em (coercivity w.r.t.\ $u$)}
\begin{enumerate}
\item[(c1)] if $p<\infty$, there exists $\alpha>0$ and $\gamma\in L^1((0,\infty))$ such that
\begin{equation}\label{eq:CoercivitycostInegra}
\ell_2(t,x,u)\ge \alpha |u|^p-\gamma(t)\quad\forall\,x\in\R^n,\ \forall\,u\in\R^m\mbox{ and for a.a.}\ t\ge0; 
\end{equation}
\item[(c2)]   if $p=\infty$, $U$ be  compact (and convex). 
\end{enumerate}
\end{enumerate}  
Our main results are stated in the following  theorem and its corollary. The first deals with any kind of closed state constraint (and even without, by taking $X=\R^n$), while  the second is limited to compact state constraints. In the first we have to explicitly require a coerciveness assumption with respect to the state that, in the second, is ensured by the compactness of $X$ provided that the Lipschitz condition on the state equation be satisfied is a slightly stronger form.

\begin{theorem}\label{MT:StateConstrPatternPReserving2}
 Besides all assumptions stated above,  
 suppose that the following  coercivity 
 hypotheses w.r.t.\ $x$  be satisfied:
 \begin{enumerate}
\item[]  {\em (coercivity w.r.t.\ $x$)} there exists a sequence $0<T_k\to\infty$ satisfying the following property: for any sequence $u_k\to u$ in $\cU$ and any $x_k\in\cX$ such that  
$\sup_{k\in\N}F_{T_k}(u_k,x_k)<\infty$ we have that  $x_k=x^{u_k,x_0}$ admits a subsequence converging in $\cX$. 
\end{enumerate}
 Then, 
 the family of optimal control problems $F_{T}$ in~\eqref{eq:OptimalControlProblem} is $(T_k)$-pattern preserving.
\end{theorem}

\begin{cor}[compact state constraint]\label{MT:StateConstrPatternPReserving2SC}
Besides all assumptions stated above, suppose that the state-constraint set $X$ be compact. 
Suppose also that 
there exists $q\in (1,p)$ such that hypothesis~\textnormal{(A0)}  holds in a stronger form with $M\in L_\loc^q([0,\infty))$, $N\in L_\loc^{\frac{pq}{p-q}}([0,\infty))$ and $A_H\in L_\loc^q([0,\infty))$, $B_H\in L_\loc^{\frac{pq}{p-q}}([0,\infty))$ for any compact set $H\subset \R^n$. Then, the family of state-constrained optimal control problems $F_{T}$  is $(T_k)$-pattern preserving for every $0< T_k\to \infty$. 
\end{cor}
The next sections will be devoted to discuss the assumptions (that we consider to be  tacitly satisfied throughout the rest of the  paper), prove the results and provide examples.

\section{State equations: basic Carath\'eodory-Lipschitz assumptions}
\label{sec:ODEandchar}



In this section we  discuss the basic assumptions done on the state equation, which ensure existence and uniqueness of solution  to problem~\eqref{eq:CauchyProblem1}. \setcounter{assumption}{-1} 
To this aim, given $q\in [1,\infty]$,  $m\in\N$,  we define the set of \emph{locally $(m,q)$-Sobolev functions on the time interval $[0,\tau)$}, $\tau\in(0,+\infty]$, by
\[
\Wl^{m,q}([0,\tau),\R^n)=\left \{x\in\cD'((0,\tau),\R^n)\; :\;  x_{|(0,T)}\in W^{m,q}((0,T),\R^n)\ \forall\, T\in(0,\tau)\right\},
\]
\noindent where $x_{(0,T)}$ denotes the restriction of the distribution $x$ to all test functions with compact support in $(0,T)$. When $n=1$, the codomain  $\R^n$ will be usually  omitted. Moreover,  when $m=0$ we use also the notation $L_\loc^q([0,\tau),\R^n):=\Wl^{0,q}([0,\tau),\R^n)$. We refer to Appendix \ref{subsec_locSob} for more details, properties and topologies on these spaces.



 Under assumption (A0), given $u\in \cU$, any function $f_u:[0,\infty)\times \R^n\to \R^n$ such that
 $f_u(t,x):=a(t,x)+b(t,x)u(t)$ for almost all $t\in [0,\infty)$ and all $x\in \R^n$,  
 satisfies the classical Carath\'eodory-Lipschitz conditions for existence and uniqueness of solutions.
This implies that, for any $u\in \cU$, there exists a \emph{unique locally absolutely continuous maximal solution} (see~\cite[Section I.5]{Hale}) to the Cauchy problem \eqref{eq:CauchyProblem1}
denoted by $x^{u,x_0}:\text{dom}(x^{u,x_0})\to \R^n$. 
Here, $\text{dom}(x^{u,x_0})$ denotes the effective  interval of definition of the maximal solution, and it is thus of the form $[0,\tau)$ with $\tau>0$ possibly equal to $\infty$.
In other words, for any $u\in \cU$, under hypothesis~\textnormal{(A0)} there exists a unique $x\in \Wl^{1,1}([0,\tau),\R^n)$  maximal solution 
(and thus $[0,\tau)$ denotes the maximal interval of definition), which we denote by $x^{u,x_0}$. 
When $x_0\in \R^n$ and/or  $u$ are fixed or clear from the context, we denote $x^{u,x_0}$ simply by $x^u$ or even $x$.

\section{State equations: global existence and convergence of solutions}
\label{Sec:Boundedness}

In Section \ref{sec:Preliminaries} we have introduced the space of controls  $\cU=(L^p((0,\infty),U),\tau_\cU)$, with $p\in (1,\infty]$ (see \ref{defU})  equipped with its weak$^\star$ topology.  On the other hand, the choice of the space of states $\cX$ and its topology require some preliminary considerations, because hypothesis~\textnormal{(A0)} does not imply forward-in-time \emph{global} existence of solutions, see the subsequent Example~\ref{example:NoGlobalExistence}. On the other hand, it turns out that such property is needed in our $\G$-convergence theorem (the forthcoming Theorem~\ref{thm:MainTheorem}). This lack can be overcome if the optimal control problems are formulated over a \emph{fixed and bounded} time horizon, as proven in~\cite[Lemma~4.2]{ButtDalMas82} (see also~\cite[Proposition~5.3.3]{Buttazzo89}). Unfortunately, the aforementioned result of Buttazzo and Dal Maso cannot be extended to an infinite horizon as it is (see Example \ref{example:NoGlobalExistence}).
  Therefore, besides the constitutive  hypothesis (A0), in Section \ref{sec:MainResults} we also imposed that the state equation satisfies the global existence assumption (A1) .
  

\begin{rem}
Roughly speaking, assumption (A1) requires that, if for a certain input $u$ the solution is globally defined, the same holds for  inputs ``close enough to $u$''.
A well-known constitutive sufficient condition  consists in requiring a \emph{sublinear growth behavior} with respect to $x$ of the right hand side of the differential equation in~\eqref{eq:CauchyProblem1}, 
see~\cite[Corollary 6.3]{Hale}. On the other hand, global existence as stated in assumption (A1) can hold also in much more general cases, as we will see in the sequel. For this reason besides their different nature, we prefer to keep separated  the two assumptions (A0)  and (A1). 
\end{rem}

The following example shows that (A0) does not imply (A1). Moreover, it is also a counterexample to the validity of Lemma~4.2 of Buttazzo and Dal Maso (\cite{ButtDalMas82}) on an infinite horizon. 

\begin{example}\label{example:NoGlobalExistence} 
Consider the scalar system
\[
\begin{cases}
x'=x^2u,\\
x(0)=1,
\end{cases}
\]
that satisfies (A0) since $b(x)=x^2$ is locally Lipschitz.
Let us take $u\equiv 0$ which leads to $x^{u}\equiv 1$ and consider  $u_k:[0,\infty)\to \R$ defined by
\[
u_k(t)=\begin{cases}
r_k\;&\text{if }t\in [0,\frac{1}{r_k}),\\
0 \;&\text{if }t\in [\frac{1}{r_k},\infty),
\end{cases}
\]
where $r_k$ is a sequence of positive real numbers such that $r_k\to 0$. 
We have that $u_k\wto u$ weakly$^\star$ in $L^p(0,\infty)$, for any $p\in (1,\infty]$.
Computing we have
\[
x^{u_k}(t)=\frac{1}{1-r_kt},
\]
with maximal interval of existence equal to $[0,\frac{1}{r_k})$ which is a proper subset of $[0,\infty)$. 
Then, $x^{u_k}\in W^{1,1}_{\loc}([0,\frac{1}{r_k}))$ only and, thus, for such system  assumption (A1)  does not hold (in particular, it fails for $u=0$ and $u_k$ as above).
    \end{example}
Under assumptions (A1)  we are led to choose the state space 
\begin{subequations}\label{defX}
\begin{equation}
\cX=(W^{1,1}_{\loc}([0,\infty),\R^n), \tau_\cX)
\end{equation}
with a suitable topology $\tau_\cX$.
To compute the $\Gamma$-limit, the chosen topologies on $\cU$ and $\cX$  must be strong enough to guarantee $\Gamma$-convergence to the sequence of joint functionals and weak enough to ensure compactness (precisely, equi-coerciveness of the joint functionals). Under suitable growth assumptions, a natural choice would be to play with weak (or weak*) topologies. If they are good enough for the control space, on the other hand, the weak topology of $W_{\loc}^{1,1}$ is too weak to pass to the limit in the state equations. 
In fact, it is known (see, e.g., \cite[Proposition 5.3.3]{Buttazzo89}) that a sufficient condition that allows to do that is the (local) uniform convergence of states. On the other hand, it is well known that the immersion $W^{1,1}_{\loc}\subset L_{\loc}^\infty$ is not compact and the weak convergence in $W_{\loc}^{1,1}$ does not imply (local) uniform convergence.   
A way to overcome this difficulty is to choose 
\begin{equation}
    \mbox{$\tau_\cX$ as the \emph{strong topology of} $L_{\loc}^\infty([0,\infty),\R^n)$}
    \end{equation}
\end{subequations}
(as done, for instance, in \cite{BC89} in the finite-horizon case) together with a suitable equi-coercivity assumption on the cost functionals that compensate the lack of compactness of the aforementioned immersion. 
Since the strong topology of $L_{\loc}^\infty([0,\infty),\R^n)$ is metrizable (see Appendix \ref{subsec_locSob})  then $\cX$ is a metric space. 
As done in the case of  $\cU$, also the convergence $\tau_\cX$ in  $\cX$ will be denoted by the symbol $\to$.




The choice of topologies on the state and control spaces confers to the set $\AI$ of admissible pairs (see \eqref{eq:FeasibleLinearized})  the following    $\Gamma$-closure property.

\begin{lemma}[$\G$-closure]\label{lemma:GammaCOnvergenceofSolutions} 
Under assumptions {\em(A0)} and  {\em (A1)} we have
\begin{enumerate}[label={\arabic*.},leftmargin=*]
\item for every $(u_k,x_k)\to (u,x)$ such that $(u_k,x_k)\in \AI$ for infinitely many $k\in\N$ we have $(u,x)\in \AI$;
\item  for every $(u,x)\in \AI$ and every  $u_k\to u$ in $\cU$ there exists $x_k \to x$ in $\cX$ such that $(u_k,x_k)\in \AI$ for all $k$ large enough.
\end{enumerate}
\end{lemma}

\begin{rem}
With the notation of double $\G$-limits (see~\cite{ButtDalMas82}), Lemma \ref{lemma:GammaCOnvergenceofSolutions} states that
\[
\chi_{\AI}=\Gamma_{\rm seq}(\cU,\,\cX^-)\lim_{k\to \infty}\chi_{\AI}.
\]
 The proof is a simple exercise  (see~\cite[Example 2.1]{ButtDalMas82}).
\end{rem}

\begin{proof}
To prove $\emph{1.}$, let us consider $(u_k,x_k)\to (u,x)$ such that $(u_k,x_k)\in \AI$ for infinitely many $k\in\N$. By definition of $\AI$, this means that there exists a subsequence (not relabeled) such that   $x_k=x^{u_k,x_0}$ for all $k\in \N$.
Moreover, the sequence $x_k:=x^{u_k,x_0}$ is converging to $x\in W^{1,1}_{\loc}([0,\infty),\R^n)$ strongly in $L^\infty_{\loc}([0,\infty),
\R^n)$, i.e., uniformly on compact sets. 
The claim will be proven by showing that  $x=x^{u,x_0}$.

To this aim it is convenient to write the Cauchy problem~\eqref{eq:CauchyProblem1} in the form
\begin{equation*}
\hspace{2.6cm}
\begin{cases}
x'(t)=f(t,x(t),u(t))\\
x(0)=x_0,
\end{cases}
\end{equation*}
where $f(t,x,u):=a(t,x)+b(t,x)u$ 
and observe that our assumptions  allow to pass to the limit in the state equation in the weak sense of distributions. 
 Indeed, for any test function $\varphi\in {\mathcal D}((0,\infty),\R^n)$  we have
\begin{equation}\label{eq:ConvDerivative}
 \begin{aligned}
&\int_0^\infty \varphi(t)\cdot\big(f(t,x_k(t),u_k(t))-f(t,x(t),u(t))\big)\,dt=\\
&\qquad=\int_0^\infty\varphi(t)\cdot\big(a(t,x_k(t))-a(t,x(t))\big)\,dt+\int_0^\infty \varphi(t)\cdot \big(b(t,x_k(t))-b(t,x(t)\big)u_k(t)\,dt+\\&\qquad\hspace{3ex}+\int_0^\infty\varphi(t)\cdot b(t,x(t))\big(u_k(t)-u(t)\big)\,dt.
 \end{aligned}
 \end{equation}
 Let $I$ be a bounded interval containing $\text{supp}\,{\varphi}$. 
 Since the sequence $x_k$  strongly converges (and it is thus bounded) in $L^\infty_{\loc}([0,\infty),\R^n)$, there exists a compact set $H\subset \R^n$ such that $x_k(t)\in H$ for all $k\in \N$ and all $t\in I$.
 About the first term on the  right hand side 
 in~\eqref{eq:ConvDerivative}, by Lipschitz continuity (assumption~\textnormal{(A0)})  we have
\[
\Big|\int_0^\infty\varphi(t)\cdot(a(t,x_k(t))-a(t,x(t)))\,dt\,\Big |\leq \|\varphi\|_\infty\int_I A_H(t)|x_k(t)-x(t)|dt\to 0\ \mbox{ as }k\to\infty,
\]
because $x_k\to x$ uniformly on $I$ 
and $A_H\in L^1_{\rm loc}([0,\infty))$. 
The second term is also converging to $0$, indeed,
 \[
\Big|\int_0^\infty \varphi(t) \cdot (b(t,x_k(t))-b(t,x(t))u_k(t)\,dt\Big|\leq  
\|\varphi\|_\infty\sup_{t\in I}|x_k(t)-x(t)|\int_I B_H(t)|u_k(t)| dt\to 0,
\]
since  $x_k\to x$ uniformly in $I$ and the integral over $I$ is bounded because $B_H\in L^{p'}_{\rm loc}([0,\infty))$, $(u_k)$ is bounded in $L^p((0,\infty),U)$  and, by H\"older's inequality, we have
$$
\int_I B_H(t)|u_k(t)| dt\le \|{B_H}_{\vert I}\|_{p'}\|{u_k}_{\vert I}\|_p.
$$
For the last term in~\eqref{eq:ConvDerivative}
let us first note that, by (1)\ of hypothesis~\textnormal{(A0)} and since $x(t)\in H$ for all $t\in I$,
there exists $\wt b\in L^{p'}(I)$ such that
\[
|b(t,x(t))|\leq \wt b(t)\;\;\;\forall \,t\in I.
\]
Then, we have
\[
\left |\int_0^\infty\varphi(t) \cdot b(t,x(t))(u_k(t)-u(t))\,dt\right |\leq\|\varphi\|_\infty\int_I\wt b(t) (u_k(t)-u(t))\,dt \to 0
\]
because $u_k\to u$ in $\cU$ that is, by the chosen topology, $u_k\weaks u$ in $L^p((0,\infty),U)$.

Since we have also $x_k(0)= x_0=x(0)$ for any $k\in \N$, 
 we have thus proved that $x$ is solution, in the sense of distributions, 
 to the Cauchy problem \eqref{eq:CauchyProblem1}. Since the latter admits a unique solution, we
 have $x=x^{u,x_0}$
 and the proof of {\it 1.}\ is concluded.

To prove \emph{2.}, let us consider $(x,u)\in \AI$, i.e., $x=x^{u,x_0}\in W^{1,1}_{\loc}([0,\infty),\R^n)$, a sequence $u_k\to u$ in $\cU$ and the corresponding sequence of local solutions   $x_k:=x^{u_k,x_0}$. 
 By assumption (A1), there exists $\overline k\in \N$ such that $x_k$ is global for $k\geq \overline k$, i.e., $x_k\in  W^{1,1}_{\loc}([0,\infty),\R^n)$ for all $k\geq \overline k$. Hence,  $(u_k,x_k)\in \AI$, for all $k\geq \overline k$. 
It remains to prove that $x_k\to x$  in $\cX$, i.e., uniformly on all compact intervals of the form $[0,T]$. 
Let us take a compact interval $I=[0,T]$ and  let $r_0:=\|x_{\vert_I}\|_\infty$.
For any $r>r_0$ and $k\geq \overline k$  we define
\[
t_k:=\sup\left\{t\in  I\;:\;|x^{u_k,x_0}(s)|\leq r\;\forall s\in [0,t]\right \}.
\]
Let us note that, by continuity, if $t_k< T$ then $|x_k(t_k)|=r$.
Let us show that, actually,  $t_k=T$ for all $k$ large enough. 
By definition, for any $t\leq t_k$ we have $x_k(t)\in H:=\overline{\B(0,r)}$. Thus,  using hypothesis~\textnormal{(A0)} 
as done in the proof of point \emph{1.}, 
we obtain
\[
|x_k(t)-x(t)|\leq \int_0^t (A_H(s)+B_H(s)|u_k(s)|\,)|x_k(s)-x(s)|\,ds+\Big|\int_0^t \wt b(s)\big(u_k(s)-u(s)\big)\,ds\Big|\quad \forall\,t\leq t_k, 
\]
with 
$A_H+B_H|u_k|$ bounded in $L^1(I)$  and $\wt b\in L^{p'}(I)$.
 Since $u_k\to u$ in $\cU$ (i.e., $u_k\weaks u$ in $L^p((0,\infty),U)$) we have 
\[
K_k:=
\Big|\int_0^t \wt b(s)\big(u_k(s)-u(s)\big)\,ds\Big|\to 0\quad\mbox{ as }k\to\infty. 
\]
On the other hand, using Grönwall's lemma  we have 
\begin{equation}\label{eq:GronwallBound}
|x_k(t)-x(t)|\leq 
MK_k\quad\forall\, t\leq t_k,
\end{equation}
with 
$$
M=\sup_{k\ge\overline k}{\rm e}^{\int_I (A_H+B_H|u_k|) \,ds}<\infty.
$$ We can thus choose $\overline k_1\in \N$ large enough such that
\[
|x_k(t)-x(t)|\leq \frac{r-r_0}{2} \quad\forall \,k\geq \overline k_1\;\;\;\forall\;t\leq t_k.
\]
This implies that
\[
|x_k(t_k)|\leq |x(t_k)|+|x_k(t_k)-x(t_k)|\leq r_0+\frac{r-r_0}{2}=\frac{r+r_0}{2}<r.
\]
By definition of $t_k$, this implies $t_k=T$  for all $k\geq \overline k_1$. Then~\eqref{eq:GronwallBound} implies that $x_k\to x$ uniformly in $L^\infty([0,T],\R^n)$, as required. By arbitrariness of $T>0$, recalling Proposition~\ref{prop_stchmp},  we have that $x_k\to x$ in $L_{\loc}^\infty([0,\infty),\R^n)$, i.e.,  $x_k\to x$ in $\cX$, which concludes also the proof of~\emph{2.}\ and of the lemma.  
\end{proof}

\section{Intrinsically coercive control systems}\label{sec:ParticularCases}

In proving  Lemma~\ref{lemma:GammaCOnvergenceofSolutions}, 
we have endowed $\cX=W^{1,1}_{\loc}([0,\infty),\R^n)$ with the strong topology of  $L^\infty_{\loc}([0,\infty),\R^n)$.
As already remarked, in general, this particular choice of the topology would require an additional equi-coercivity assumption  to guarantee sequential compactness of minimizers (see Definition~\ref{defn:EquiCoercivityGeneral} and Theorem~\ref{lemma:GammaCOnvMAINProp}), namely,  for every $0<T_k\to\infty$, 
\begin{equation}\label{eq:EquiCoercivty}
\cF_{T_k}(u_{T_k},x_{T_k})\le C\ \implies (u_{T_k},x_{T_k})\to(u,x)\mbox{ in }\cU\times \cX\mbox{ (up to a subsequence).}
\end{equation}
  In this section we illustrate particular yet remarkable cases in which the global existence assumption (A1) holds and, moreover, 
the convergence of the states in the strong $L^\infty_{\loc}$ sense holds independently of the cost functional, and thus,  without any additional assumption, \emph{coercivity w.r.t.\ $x$} holds as formally defined in the hypothesis of Theorem~\ref{MT:StateConstrPatternPReserving2}.

\subsection{Uniformly bounded systems}

Given a non-empty, convex and closed set $U\subseteq \R^m$, in this subsection we consider $\cU=L^\infty((0,\infty),U)$ (i.e., $p=\infty$) endowed with the weak$^\star$ topology. 
Let us introduce the following definition of uniform boundedness.
\begin{defn}\label{defn:UniformBoundedness}
Given $\varnothing\ne U\subseteq \R^m$ and $a,b:[0,\infty)\times \R^n\to \R^n$ satisfying hypothesis~\textnormal{(A0)},  system~\eqref{eq:CauchyProblem1} is said to be  \emph{uniformly globally bounded with respect to $U$} if for any  $x_0\in \R^n$ and any bounded set $\cB \subset L^\infty((0,\infty),U)$
there exists a bounded set $X_{\cB,x_0}\subset \R^n$ such that $x^{u,x_0}(t)\in X_{\cB,x_0}$ for every $u\in \cB$ and every $t\in \textit{\rm dom}(x^{u,x_0})$.  
When $U=\R^m$ we simply say that the system is \emph{uniformly globally bounded}.
\end{defn}

\begin{rem}
We note that, when the set $U\subset \R^m$ is bounded, this is equivalent to say that, for any $x_0\in \R^n$ there exists a bounded set $X_{x_0}\subset \R^n$ such that $x^{u,x_0}(t)\in X_{x_0}$  for every $u\in L^\infty((0,\infty),U)$ and every $t\in \text{dom}(x^{u,x_0})$.
\end{rem}
Uniform global boundedness ensures that for any initial condition $x_0\in \R^n$ and for any bounded set of control inputs (satisfying the constraint $u(t)\in U$ almost everywhere), the values of the corresponding solutions are confined in a bounded subset of $\R^n$.
Of course, if we are only interested in a specific initial condition $x_0\in \R^n$, the quantifier ``for any $x_0\in \R^n$'' can be avoided. The notion of boundedness in Definition~\ref{defn:UniformBoundedness} has been introduced in various contexts, possibly  with slightly different nomenclatures/def\-i\-ni\-tions,  see for example~\cite{SontagWang96},~\cite{Bacc98},~\cite[Definition 2.4]{Miro23} and references therein. Historically, it was also referred to as ``Lagrange stability'', see for example~\cite[Section 1.5]{BhatiaSzego}.

Moreover, in control theory, there is a well-known ``stability property'' related to Definition~\ref{defn:UniformBoundedness}, known as \emph{input-to-state stability} (shortly ISS), see~\cite{Sontag94} and the recent monograph~\cite{Miro23}.
The notion of ISS has been introduced in~\cite{Sont89}  in the context of stabilization problems. It is now considered a central property in the context of nonlinear control systems, also due to its elegant characterization in terms of existence of Lyapunov functions, as provided in~\cite{Sontag94}. In particular, the ISS property provides a sufficient condition (which in some cases is also necessary) for Definition~\ref{defn:UniformBoundedness}, as proved in~\cite[Section 2.5]{Miro23}. 

Let us also note the relations of Definition~\ref{defn:UniformBoundedness} with the notion of forward invariant set. Formally, a set $X\subset \R^n$ is said to be \emph{forward invariant (with respect to $U\subseteq\R^m$)}, if for any  $x_0\in X$  and any $u\in L^\infty((0,\infty),U)$ it holds that $x^{u,x_0}(t)\in X$ for all $t\in [0,\infty)$. It is now easy to see that if system~\eqref{eq:CauchyProblem1} admits a \emph{bounded} forward invariant set $C$  (with respect to $U\subseteq\R^m$), the condition introduced in Definition~\ref{defn:UniformBoundedness} is trivially satisfied for all $x_0\in C$. 

Definition~\ref{defn:UniformBoundedness} is particularly beneficial in our context since in certain cases it is a ``verifiable'' sufficient condition for assumption (A1)  and, moreover, when $u_k\to u$ in $\cU$ the corresponding solutions will converge   uniformly on compact sets, as formalized in the following lemma.

\begin{prop}\label{Lemma:LocallyLispchitzSol}
 Suppose that {\em(A0)} be satisfied and let  $q\in (1,\infty]$. If system~\eqref{eq:CauchyProblem1} is uniformly globally bounded w.r.t.\  a non-empty, closed and convex set $U\subseteq \R^m$,  then {\em(A1)} is satisfied.
Suppose further that hypothesis {\em(A0)} be satisfied in a stronger form with $M,N\in L_\loc^q([0,\infty))$ and $A_H,B_H\in L_\loc^q([0,\infty))$ for any compact set $H\subset \R^n$. Then the following property holds:
\begin{itemize}[leftmargin=*]
    \item[$\empty$] if $u_k\to u$ in $\cU$ (i.e., in $L^\infty((0,\infty),U)$ with its weak$^\star$ topology) then 
\[
x^{u_k,x_0}\to x^{u,x_0}\;\;\;\text{weakly in } W^{1,q}_{\loc}([0,\infty),\R^n),
\]
and, thus, in $\cX$ (i.e., strongly in $L^\infty_{\loc}([0,\infty),\R^n)$).
\end{itemize}
\end{prop}
\begin{proof}
Suppose that system~\eqref{eq:CauchyProblem1} be uniformly globally bounded w.r.t.\   $U$.  Consider any $x_0\in \R^n$ and any $u\in L^\infty((0,\infty),U)$. By Definition~\ref{defn:UniformBoundedness}, there exists a bounded set $X_{u,x_0}\subset \R^n$ such that $x^{u,x_0}(t)\in X_{u,x_0}$ for all $t\in [0,\infty)$, and it is thus defined on $[0,\infty)$. This implies that hypothesis~\textnormal{(A1)} holds.

Suppose now that hypothesis~\textnormal{(A0)} be satisfied with $M,N\in L_\loc^q([0,\infty))$ and $A_H,B_H\in L_\loc^q([0,\infty))$ for any compact set $H\subset \R^n$.
Consider any $u\in \cU$ and any sequence  $u_k\to u$ in $\cU$. Since, in particular, the sequence $u_k$ is bounded in $\cU=L^\infty((0,\infty),U)$, again by Definition~\ref{defn:UniformBoundedness}, there exists $H\subset \R^n$ bounded such that $x^{u_k,x_0}(t)\in H$, for all $k\in \N$ and all $t\in [0,\infty)$ (without loss of generality, we can suppose that $H$ be compact and $0\in H$).  This implies  that the sequence $x^{u_k,x_0}$ is bounded in $L^\infty((0,\infty),\R^n)$.
 Since (A0) is satisfied, it holds that
\begin{equation}\label{eq:BoundednessofDerivative}
\Big|\frac{d}{dt}x^{u_k,x_0}(t)\Big|
\leq |a(t,x^{u_k,x_0}(t))|+|b(t,x^{u_k,x_0}(t))||u_k(t)|\leq \widetilde M(t)
\end{equation}
with $\widetilde M(t)=M(t)+A_H(t)D+(N(t)+B_H(t)D)u_{M}$ where $D=\text{diam}(H)\geq 0$ and $u_M\geq 0$ is such that $\|u_k\|_\infty\leq u_M$ for all $k\in \N$.
Since we are supposing $A_H,B_H,M,N\in L_\loc^q([0,\infty))$ this implies that $\widetilde M\in L_\loc^q([0,\infty))$.

Hence,  the sequence $x^{u_k,x_0}$ is bounded in $\Wl^{1,q}([0,\infty),\R^n)$. Then, it admits a weakly$^\star$ converging subsequence $x_k\wto x$ in $\Wl^{1,q}([0,\infty),\R^n)$.
By the Uryshon's  property stated in Theorem~\ref{th_Ury_Wloc} it suffices to show that $x=x^{u,x_0}$. This equality can be obtained with the same arguments of the proof of \emph{1.}\ of Lemma~\ref{lemma:GammaCOnvergenceofSolutions} and it is thus not further developed here. The last statement directly follows by  Lemma~\ref{lem_ucbs}.
\end{proof}

\subsection{Linear control systems}
In this section we allow $p\in (1,\infty]$ and $\cU$ as in \eqref{defU}, as usual.
We consider \emph{non-autonomous linear control systems} and show that, under certain hypotheses, solutions corresponding to converging controls also converge in $\cX$ (see \eqref{defX}).

Given $A:[0,\infty)\to \R^{n\times n}$ and $B:[0,\infty)\to \R^{n\times m}$ Lebesgue measurable (matrix-valued) functions, we consider the linear system defined by
\begin{equation}\label{eq:LinearControlSystem}
\begin{cases}
x'(t)=A(t)x(t)+B(t)u(t),\\
x(0)=x_0\in \R^n.
\end{cases}
\end{equation}
Let us introduce the following standard definition.
\begin{defn}[state-transition matrix]
Given $A\in L^1_{\rm loc}([0,\infty),\R^{n\times n})$ and any $t_0\in [0,\infty)$ we define
$\Phi(\cdot, t_0):(0,\infty)\to \R^{n\times n}$ as the unique solution to the (matrix-valued) Cauchy problem
\[
\begin{cases}
 M'(t)=A(t)M(t),\\
M(t_0)=\text{I}_n,
\end{cases}
\]
where $I_n$ is the identity matrix of order $n$.
The function $\Phi:[0,\infty)\times [0,\infty)\to \R^{n\times n}$
is called the \emph{state-transition matrix} associated to $A$. 
\end{defn}
For an introduction to state-transition matrices for linear systems, we refer to~\cite[Chapter 4.6]{Khalil02} and references therein.

Let us suppose that $B\in L^{p'}_{\rm loc}([0,\infty),\R^{n\times m})$. Given any $u\in \cU$ and any $x_0\in \R^n$, the linear system~\eqref{eq:LinearControlSystem} admits a \emph{unique global solution}, given by
\begin{equation}\label{eq:SOlutionLinearsystems}
x^{u,x_0}(t)=\Phi(t,0)x_0+\int_0^t\Phi(t,s)B(s)u(s)\,ds,\quad \,t\in [0,\infty).
\end{equation}
The latter is also known as \emph{variation of constants/parameters formula}, see for example~\cite[Chapter~3]{Hale}.
 Therefore, (A1) is trivially satisfied.

Let us now study the convergence of solutions as $u_k\to u$ in $\cU$.

\begin{prop}\label{prop_bbs_lin}
Consider $A\in L_\loc^\infty([0,\infty),\R^{n\times n})$ and $B\in L_\loc^s([0,\infty),\R^{n\times m})$ for some $s\in (1,\infty]$. Let us suppose that the solution $q\in [0,\infty]$ of the equation
\begin{equation*}\label{eq:HolderConjugacy}
\frac{1}{p}+\frac{1}{s}=\frac{1}{q}
\end{equation*}
is such that $q>1$.
If $u_k\to u$ in $\cU$ then  
\[
x^{u_k,x_0}\to x^{u,x_0}\;\;\;\text{weakly$^\star$ in } W^{1,q}_{\loc}([0,\infty),\R^n),
\]
and, thus, in $\cX$ (i.e., strongly in $L^\infty_{\loc}([0,\infty),\R^n)$).
\end{prop}
\begin{proof}
First of all we note that, since $A\in L_\loc^\infty([0,\infty),\R^{n\times n})$, for every $T>0$ there exists $\lambda_T\geq 0$ such that $|A(t)|\leq \lambda_T$ for a.a.\ $t\in [0,T)$. This implies, by   Grönwall's  inequality (see for example~\cite[Lemma 2.7]{Teschl12}) that for every $T>0$ we have
\begin{equation}\label{eq:BoundTransitionMatrix}
|\Phi(t,s)|\leq e^{\lambda_T(t-s)},\;\;\forall s\leq t\leq T.
\end{equation}
We explicitly develop the case $p,s,q<\infty$: the remaining cases are similar and left to the reader.
Consider any $u\in \cU$, any $x_0\in \R^n$ and any $T>0$.  Using~\eqref{eq:SOlutionLinearsystems} we have
\[
\int_0^T |x^{u,x_0}(t)|^q\,dt\leq \int_0^T|\Phi(t,0)x_0|^q\;dt+\int_0^T \int_0^t|\Phi(t,s)B(s)u(s)|^q\,ds.
\]
By~\eqref{eq:BoundTransitionMatrix}, the first term on the right-hand side is  bounded by $\frac{|x_0|^q}{q\lambda_T}(e^{q\lambda_T T}-1)$. For the second one we note that
\begin{equation}\label{eq:Bound0LinearProof}
\begin{aligned}
\int_0^T \int_0^t |\Phi(t,s)B(s)u(s)|^q\,ds\,dt&\leq  \int_0^T \int_0^te^{q\lambda_T (t-s)}|B(s)u(s)|^q\,ds\;dt\\&\leq   \int_0^T e^{q\lambda_Tt} \int_0^T|B(s)|^q|u(s)|^q\,ds\,dt\\&\leq \Lambda \int_0^T |B(t)|^q \,|u(t)|^q\,dt
\end{aligned}
\end{equation}
with $\Lambda=\int_0^T e^{q\lambda_T  t}\,dt$. 
Now, by H\"older's inequality with conjugate  exponents $\frac{s}{q}$ and $\frac{p}{q}$, we have 
\[
\int_0^T |B(t)|^q|u(t)|^q\,dt\leq \left(\int_0^T|B(t)|^s\,dt\right)^{\frac{q}{s}} \left(\int_0^T|u(t)|^p\,dt\right)^{\frac{q}{p}}.
\]
Since $B\in L_\loc^s([0,\infty),\R^{n\times m})$ we have thus proven that, for any $T>0$ there exist $M_T,N_T\geq 0$ such that
\begin{equation}\label{eq:Bound1LinearProof}
\|(x^{u,x_0})_{\vert (0,T)}\|_q^q\leq M_T\|u\|^q_p+N_T|x_0|^q\;\;\;\forall u\in L^p((0,\infty),U).
\end{equation}
 Moreover, since
\[
\frac{d}{dt}x^{u,x_0}(t)=A(t)x^{u,x_0}(t)+B(t)u(t),
\]
given any $T>0$ we have 
\begin{equation}\label{eq:Bound2LinearProof}
\begin{aligned}
\|(x^{u,x_0})'_{\vert (0,T)}\|_q^q&=\int_0^T \Big|\frac{d}{dt}x^{u,x_0}(t)\Big|^q dt\leq 
 \|A_{\vert (0,T)}\|^q_\infty\|x^{u,x_0}_{\vert (0,T)}\|^q_q+\int_0^T |B(t)|^q |u(t)|^q\,dt
\\&\leq \|A_{\vert (0,T)}\|^q_\infty\|x^{u,x_0}_{\vert (0,T)}\|^q_q+\|B_{\vert (0,T)}\|_s^q\|u\|^q_p,
\end{aligned}
\end{equation}
for all $u\in \cU$.

Consider now a converging sequence $u_k\to u$ in $\cU$, which is in particular bounded. By~\eqref{eq:Bound1LinearProof} and~\eqref{eq:Bound2LinearProof} we directly have that the sequence $x^{u_k,x_0}$ is bounded in $\Wl^{1,q}([0,\infty),\R^n)$.  Then, it admits a  weakly* converging subsequence $x_k\wto x$ in $\Wl^{1,q}([0,\infty),\R^n)$.
 By the Uryshon's  property stated in Theorem~\ref{th_Ury_Wloc} the claim follows by showing that  $x=x^{u,x_0}$. This can be done by using the same arguments of the proof of  \emph{1.}\ of Lemma~\ref{lemma:GammaCOnvergenceofSolutions}.  The application of  Lemma~\ref{lem_ucbs} concludes the proof.  
\end{proof}
\begin{rem}
Optimal control for linear models as in~\eqref{eq:LinearControlSystem} is a classical topic in  control theory. Usually, such state equation is coupled with a quadratic cost functional of the form
\begin{equation}\label{eq:LQRCost}
\cJ_T(x,u)=\int_0^T x^\top(t)Q(t)x(t)+u^\top(t)R(t)u(t)\,dt
\end{equation}
with measurable symmetric matrix-valued functions $Q:[0,\infty)\to \R^{n\times n}$, $R:[0,\infty)\to \R^{m\times m}$ satisfying uniform positive-definiteness conditions, i.e.
\[
a_1 I_n\preceq Q(t)\preceq a_2I_n\;\;\text{ and }\;\;a_1 I_m\preceq R(t)\preceq a_2I_m\;\;\text{for almost all }t\ge0,
\]
for some $0<a_1<a_2$. Such problem is also known as \emph{linear-quadratic regulator} (LQR) problem, see \cite[Chapter 6]{Liberzon12} and references therein. With cost functionals as in~\eqref{eq:LQRCost}, it is natural to consider, in the hypothesis of Proposition~\ref{prop_bbs_lin}, $s=\infty$ and $p=q=2$. 
\end{rem}

\section{Optimal control problems with increasing time horizons}\label{Sec:StateConstr}

This 
section is basically devoted to prove our main results, Theorem \ref{MT:StateConstrPatternPReserving2} and its Corollary \ref{MT:StateConstrPatternPReserving2SC}. The road map for the proof is tracked by the Pattern Preservation Theorem \ref{prop:COnvergenceofOptima}. Following it, besides ensuring compactness (i.e., equi-coercivity),  we have to prove a suitable 
$\Gamma$-convergence result. Since the latter is interesting also by itself as well, we prefer to present it in a  separated  statement.

Let us, then, consider  (sequences of) general 
optimal control problems  as defined  and under the as\-sump\-tions of Section \ref{sec:MainResults}. 
In particular, they may involve \emph{state constraints}.
Just to exemplify, it can be required  that the state trajectory  satisfy   a constraint of the form $x\in X$  with $X$ a closed and proper subset of $\R^n$. Constraints of this kind arise naturally in many applications in which the system's state must remain within a \emph{feasible region}, in order to satisfy safety, physical, or regulatory limits, see the overview in~\cite{Frankowska2010}. Such viability requirement is natural in control of population/epidemic dynamics, as we will present by  an example in the subsequent Section~\ref{sec:Examples}. 
As already remarked in 
Section \ref{sec:MainResults} (see \eqref{mr:sc}), 
the aforementioned state constraints, when incorporated in the cost functional,   
results in an additive  contribution of the form $\chi_{X}(x)$ that is, at most, lower semicontinuous. This fact 
causes a lack of continuity which prevents 
a direct application of the $\G$-convergence results of \cite{F2000} (see also \cite{BF98}).

Let us recall that the control space $\cU=L^p((0,\infty),U)$  with  $p\in(1,\infty]$ (where $U=\R^m$ if  $p<\infty$ and $U$
     a non-empty, closed and convex subset of  $\R^m$ if  $p=\infty$), is  endowed by  the weak$^\star$ topology, while  the space of states $\cX=\Wl^{1,1}([0,\infty),\R^n)$ is equipped with  the strong topology of $L^\infty_{\loc}([0,\infty),\R^n)$.
We now re-cast the constrained optimal control problem 
stated   in~\eqref{eq:OptimalControlProblem} in a form more suitable to our purposes, considering the problems
\begin{equation}\label{eq:AbstractOptContrProb00}
\begin{aligned}
\begin{array}{c}
\displaystyle\inf
\cJ_T(u,x)\\[-1ex]
\ \\[-1ex]
\text{subject to}\\[-1ex]
\ \\[-1ex]
(u,x)\in \cA,
\end{array}
\end{aligned}
\end{equation}
with $\cA:=\{(u,x)\in \cU\times\cX\;\vert\;x=x^{u,x_0}\}$ as before, while 
\begin{equation}\label{eq_coerick}
  \cJ_T(u,x):=\int_0^T \ell(t,x(t),u(t))\,dt+\begin{cases}\displaystyle \int_T^\infty |u(t)|^p\,dt&\text{if $p<\infty$},\\[2ex]
\displaystyle{\int_0^\infty \chi_U(u(t))\,dt} &\text{if $p=\infty$}.
\end{cases}
\end{equation}
For every $T>0$ (possibly $T=\infty$, with the convention $\int_\infty^\infty=0$) we consider the joint functional $ \cF_T:\cU\times\cX\to [0,+\infty]$  defined by
\begin{equation}\label{eq:JointFunctionalStateConstr}
\cF_T=\cJ_T+\chi_\cA.
\end{equation}
It is worth noting that problems~\eqref{eq:OptimalControlProblem} and~\eqref{eq:AbstractOptContrProb00} are equivalent. This is obvious in the case $p=\infty$ because 
the additional term $\int_0^\infty \chi_U(u(t))\,dt$ is $0$ and plays only a formal role. In the case $p<\infty$,  
if $(u,x)\in\cU\times\cX$ is a feasible (resp. optimal) pair for~\eqref{eq:OptimalControlProblem}, it suffices to consider $(\wt u,\wt x)\in \cU\times \cX$ defined by 
\[
\wt u(t):=\begin{cases}
u(t)\;\;\;&\text{if }t\leq T,\\
0 \;\;\;&\text{if }t> T,
\end{cases}
\]
and $\wt x=x^{\wt u,x_0}$. It then holds that $x(t)=\wt x(t)$ for all $t\leq T$ and
it is then easy to see that $(\wt x,\wt u)$ is again a feasible (resp. optimal) pair for~\eqref{eq:OptimalControlProblem} as well as a feasible (resp. optimal) pair for~\eqref{eq:AbstractOptContrProb00}.
Conversely, if $(u,x)\in \cU\times\cX$ is feasible (resp. optimal) for~\eqref{eq:AbstractOptContrProb00} it is clearly  feasible (resp. optimal) also for~\eqref{eq:OptimalControlProblem}.

\begin{rem}\label{rem_efc}
In the case $p<\infty$, the equivalent formulation \eqref{eq_coerick} has the advantage that  if the integrand $\ell_2$ satisfies a polynomial growth condition of order $p$ from below, then the same holds for the integrand on the whole infinite horizon. This would ensure equi-coercivity w.r.t.\ $u$. Of course, a similar trick does not work to guarantee coercivity also w.r.t.\ $x$. Indeed, an integral contribution on $(T,\infty)$ involving the state variable would lead to a different problem (in general, not equivalent to the original one).         
\end{rem}


Before providing 
the first result
of this section,
we state a useful lemma.
\begin{lemma}[tails replacement]\label{lemma:Useful_L1}
 Let us consider $u\in \cU$, $u_k \to u$ in $\cU$ and any sequence of increasing positive numbers $T_k\to \infty$. 
Given any  sequence $w_k$ bounded in $\cU$ and the sequence $\widetilde u_k\in \cU$ defined by
\[
\widetilde u_k(t)=\begin{cases}
u_k(t)\;\;\;\;\;&\text{if } t\leq T_k,\\
w_k(t)\;\;\;\;\;\;&\text{if } t> T_k,
\end{cases}
\]
we have that $\widetilde u_k\to u$ in $\cU$.
\end{lemma}

\begin{proof} Recalling that $\cU$ is endowed with the weak$^\star$ topology, and  weakly$^\star$ converging sequences are bounded, we have that the sequence $u_k$ is bounded in $L^p((0,\infty),\R^m)$. Moreover, since  the sequence $w_k$ is by assumption bounded, also the sequence $\widetilde u_k$ is bounded.
Then, there exists a subsequence of $\widetilde u_k$ (not relabeled) that converges to some $v$ in $\cU$. On the other hand, by taking a test function $\varphi\in\cD((0,\infty),\R^m)$, by compactness of the support of $\varphi$, for every $k$ large enough  we have 
$$
\int_0^{\infty}\widetilde u_k\cdot\varphi\,dt=\int_0^{\infty}u_k\cdot\varphi\,dt\to \int_0^{\infty}u\cdot\varphi\,dt.
$$
Thus $\widetilde u_k\wto u$ in $\cD'((0,\infty),\R^m)$ and, by uniqueness of the limit, we obtain $v=u$. Since on bounded sets the weak$^\star$ topology is metrizable, this implies that the whole sequence $\widetilde u_k$ weakly$^\star$ converges to $u$, as claimed.
\end{proof}


We are now in position to state the main results of this section under the assumptions of  Section~\ref{sec:MainResults}. Classically, the strategy of the proof consists in
checking the \emph{liminf inequality} and \emph{recovery sequence} conditions  
of Remark~\ref{rem_lirs}.

 \begin{thm}[$\G$-convergence]\label{thm:MainTheorem} 
Under assumptions \emph{(A0)}, \emph{(A1)}, \emph{(a)} and \emph{(b)} of Section \ref{sec:MainResults}, we have 
\[
\G_{\rm seq}^-(\cU\times\cX)\lim_{T\to\infty}\cF_{T}=\cF_\infty,
\]
where $\cF_T$ and $\cF_\infty$ are the joint functionals introduced in~\eqref{eq:JointFunctionalStateConstr}.
\end{thm}

\begin{proof} Let us prove that (see Remark \ref{rem_GammaT}), for any sequence $0< T_k\to \infty$,  we have 
\[
\G_{\rm seq}^-(\cU\times\cX)\lim_{h\to\infty}\cF_{T_h}=\cF_\infty.
\]
To this aim, let us start by proving the liminf inequality. 
Consider a sequence  $u_k\to u$ in $\cU$ and $x_k\to x$ in $\cX$ (i.e., strongly in $L^\infty_{\loc}([0,\infty),\R^n)$). Without loss of generality, we can suppose that 
$$
\liminf_{k\to\infty}\cF_{T_k}(u_k,x_k)<+\infty
$$
since, otherwise, there is nothing to prove. 
Thus,  possibly passing to a subsequence, we can suppose that the sequence $\cF_{T_k}(u_k,x_k)$ be bounded. 
In particular, by definition of $\cF_{T_k}=\cJ_{T_k}+\chi_\cA$, we have that 
\begin{equation}\label{eq:condition}
(u_{k},x_k)\in \cA,\, \text{i.e., }x_{k}=x^{u_{k},x_0},\quad\forall \,t\in (0,T_k),
\end{equation}
for all $k\in \N$. In the case $p=\infty$, we also have $u_k(t)\in U$ for all $k\in \N$ and almost all  $t\ge0$.

Under our standing assumptions, assertion \emph{1.}\ of  Lemma~\ref{lemma:GammaCOnvergenceofSolutions} implies that $x_k=x^{u_k,x_0}\to x^{u,x_0}$ in $L^\infty_{\loc}([0,\infty),\R^n)$. By uniqueness of the limit we have $x^{u,x_0}=x$, i.e., $(u,x)\in \cA$. 
Summarizing, we have 
\begin{equation}\label{eq:COnvegencIndicatrix}
\chi_{\cA}(u,x)=\chi_{\cA}(u_k,x_k)=0\quad\forall \,k\in \N.
\end{equation}

Let us now focus on the cost functionals. Let us recall that by hypothesis \textnormal{(a)} we have $\ell(t,x,u)=\ell_1(t,x)+\ell_2(t,x,u)$, where $\ell_1$ is a non-negative normal integrand, while $\ell_2$ is a non-negative normal convex integrand. Thus, for any $T>0$ (and also in the case $T=\infty$) let us consider the functionals  $\cH_T,\cI_T:\cU\times\cX\to [0,+\infty]$ defined by  
\begin{equation*}
\cH_T(u,x):=\int_0^T \ell_1(t,x(t))\,dt\;\text{ and }\;\cI_T(u,x):=\int_0^T\ell_2(t,x(t),u(t))\,dt.
\end{equation*}
About the sequence $\cH_{T_k}$, 
the application of Fatou's Lemma gives
\[
\liminf_{k\to \infty}\cH_{T_k}(u_k,x_k)=\liminf_{k\to \infty}\int_{0}^\infty \ell_1(t,x_k(t))\mathbf{1}_{[0,T_k]}(t)\,dt\geq \int_{0}^\infty \liminf_{k\to\infty}\ell_1(t,x_k(t))\mathbf{1}_{[0,T_k]}(t)\,dt.
\]
Since $x_k\to x$ pointwisely and $\ell_1(t,\cdot)$ is lower semicontinuous, we have 
\[
\liminf_{k\to\infty}\ell_1(t,x_k(t))\mathbf{1}_{[0,T_k]}(t)\geq \ell_1(t,x(t))\text{ for a.a.}\ t\ge0.
\]
Thus, summarizing, we have proved that
\begin{equation}\label{eq:LimInfFUncFirstPart}
\liminf_{k\to \infty}\cH_{T_k}(u_k,x_k)\geq \int_0^\infty \ell_1(t,x(t))\,dt=\cH_\infty(u,x).
\end{equation}
Let us now consider the sequence $\cI_{T_h}$. 
By De Giorgi and Ioffe's Semicontinuity Theorem (see for instance \cite[Theorem 7.5]{FL07} or \cite[Section 2.3]{Buttazzo89}), for 
every $T\in(0,+\infty)$ the functional 
$\cI_T:  \cU\times\cX\to[0,\infty]$ 
is  lower semicontinuous with respect to the chosen topologies, because  $\ell_2$ is a normal convex integrand. Moreover, also $ \cI_\infty$ is  lower semicontinuous. Indeed, since the integrand is non-negative we have
\begin{equation*}
\int_0^{\infty}\ell_2(t, x(t),u(t))\,dt=\sup_{T>0}\int_0^{T}\ell_2(t,x(t),u(t))\,dt,
\end{equation*}
and  the supremum  of any collection of lower semicontinuous functionals is lower semicontinuous.

Let us now use  the greedy strategy $u_g$ introduced in hypothesis~\textnormal{(b)} to modify the sequence $u_k$ 
as follows:
\[
\wt u_k(t)=\begin{cases}
u_k(t)&\text{if }0< t\leq T_k,\\
u_g(t,x_k(t))&\text{if }t> T_k.\\
\end{cases}
\]
Since $u_g:(0,\infty)\times \R^n\to \R^m$ is Borel measurable and $x_k$ is continuous, then $\wt u_k:(0,\infty)\to\R^m$ is Lebesgue measurable for any $k\in \N$. 
Moreover,  by the local boundedness of $u_g$ (hypothesis~\textnormal{(b1)}) and since $x_k$ are bounded in $L^\infty_{\loc}([0,\infty),\R^n)$, we have that $u_g(\cdot,x_k(\cdot))$ is bounded in $L^p((0,\infty),\R^m)$ in the case $p<\infty$, and  takes values in $U$ in the case $p=\infty$ (by hypothesis~\textnormal{(b2)}) . We can thus apply Lemma~\ref{lemma:Useful_L1}, obtaining $\wt u_k\to u$ in $\cU$. Since, by hypothesis~\textnormal{(b)}, $\ell_2(t,x,u_g(t,x))=0$ for almost all $t>0$ and for all $x\in \R^n$, by the aforementioned lower semicontinuity of $\cI_\infty$  we have
\begin{equation}\label{eq:LimInfFUncSecondPart}
\begin{aligned}
\liminf_{k\to \infty}\cI_{T_k}(u_k,x_k)&=\liminf_{k\to \infty}\int_0^{T_k} \ell_2(t,x_k(t),u_k(t))\,dt=\liminf_{k\to \infty}\int_0^\infty \ell_2(t,x_k(t), \wt u_k(t))\\
&\geq\int_0^\infty\ell_2(t,x(t), u(t))=\cI_\infty(u,x).
\end{aligned}
\end{equation}
By collecting~\eqref{eq:COnvegencIndicatrix},~\eqref{eq:LimInfFUncFirstPart} and~\eqref{eq:LimInfFUncSecondPart} and using the superadditivity of the limit inferior, we obtain
\[
\begin{aligned}
 \cF_\infty(u,x)&=\cH_\infty(u,x)+\cI_\infty(u,x)\\&\leq \liminf_{k\to \infty}\cH_{T_k}(u_k,x_k)+ \liminf_{k\to \infty}\cI_{T_k}(u_k,x_k) \\&\leq \liminf_{k\to \infty}\left (\cH_{T_k}(u_k,x_k)+ \cI_{T_k}(u_k,x_k)\right)\le\liminf_{k\to \infty} \cF_{T_k}(u_k,x_k),
\end{aligned}
\]
where the last inequality follows by the non-negativity of the additional integral term in \eqref{eq_coerick}. 
This concludes the proof of the liminf inequality.

We now prove the existence of a recovery sequence, i.e., given $(u,x)\in \cU\times \cX$ we have to show that there exists a sequence $(u_k,x_k)\to (u,x)$ in $\cU\times\cX$ such that  $\dis \cF_\infty(u,x)\ge \limsup_{k\to \infty} \cF_k(u_k,x_k)$. We can suppose that $\cF_\infty(u,x)<\infty$, otherwise there is nothing to prove. This implies $(u,x)\in \cA$, i.e., $x=x^{u,x_0}$ and moreover that  $u(t)\in U$ for almost all $t>0$ in the case $p=\infty$.
Now, let us consider the constant sequence $(u_k,x_k)\equiv (u,x)$.  By non-negativity of $\ell_1$ and $\ell_2$ and using the fact that, in the case $p<\infty$, \begin{equation}\label{eq:ConvControlProof1}
 \lim_{k\to \infty}\int_{T_k}^\infty |u(t)|^p \,dt=0,
 \end{equation}
we have
\[
\begin{aligned}
\limsup_{k\to \infty} \cF_{T_k}(u,x)&=\limsup_{k\to \infty}\int_0^{T_k}\ell_1(t,x(t))+\ell_2(t,x(t),u(t))\,dt\\&\leq  \int_0^\infty\ell_1(t,x(t))+\ell_2(t,x(t),u(t))\,dt=  \cF_\infty(u,x),
\end{aligned}
\]
concluding the proof.
    \end{proof}

    \begin{rem}
 The assumption of {\em existence of a  bounded greedy control strategy}, introduced in Section~\ref{sec:MainResults} (point~\textnormal{(b)}) and used in Theorem~\ref{thm:MainTheorem},  imposes a particular structure on the running cost function $\ell$, besides its non-negativity.
A sufficient condition for the validity of \textnormal{(b1)} is that
 there exists $\alpha\in L^p(0,\infty)$  such that
$$
|u_g(t,x)|\le\alpha(t)\quad \forall\, x\in \R^n \mbox{ and for a.a.}\ t>0.
$$
As a remarkable example, one can consider an arbitrary normal integrand $\ell_1:[0,\infty)\times \R^n\to [0,\infty]$ (possibly including also an additive term $\chi_X(x)$ with $X$ a closed subset of $\R^n$) 
and a normal convex integrand $\ell_2:[0,\infty)\times \R^m\to [0,\infty]$ such that, if $p\in (1,\infty)$,
\[
\ell_2(t,0)=0,\;\;\;\text{ for a.a.}\ t\ge0.
\]
In this case, condition \emph{(b)} is satisfied by taking $u_g\equiv0$. 

If $p=\infty$,  we can instead  suppose that there exists $u_\star\in U$ such that 
\[
\ell_2(t,u_\star)=0,\;\;\;\text{ for a.a.}\ t\ge0,
\]
and define $u_g\equiv u_\star$ to satisfy condition~\textnormal{(b2)}.
    \end{rem}

We are now in a position to prove Theorem~\ref{MT:StateConstrPatternPReserving2}.

\

\noindent {\em Proof of Theorem~\ref{MT:StateConstrPatternPReserving2}.} 
The proof is an application of the Pattern Preservation Theorem \ref{prop:COnvergenceofOptima}, whose  $\Gamma$-convergence assumption {\em(2)} is satisfied by the $\Gamma$-convergence Theorem \ref{thm:MainTheorem}. So, it remains just to prove that also the equi-coercivity hypotheses {\em(1)} is fulfilled for the sequence  $T_k\to\infty$ for which the standing assumption of {\em coercivity w.r.t.\ x}  holds.

To this aim, let us consider a sequence  $(u_k,x_k)\in \cU\times \cX$ such that 
\begin{equation}\label{eq_FTkmC} \cF_{T_k}(u_k,x_k)\leq C
\end{equation}
for some $C\geq 0$. We first claim that, up to a subsequence, there exists a $u\in \cU$ such that $u_k\to u$ in $\cU$.

In the case $p=\infty$, 
$u_k(t)\in U$ for all $k\in \N$ and a.e.\  $t>0$. Since $U$ is compact and convex by hypothesis (c2) of Section~\ref{sec:MainResults}, then  $u_k\to u$ in $\cU$ (up to a subsequence), for a certain $u\in \cU$.
Let us now suppose that $p\in (1,+\infty)$ and condition~\eqref{eq:CoercivitycostInegra} holds.  Since $\cF_{T_k}(u_k,x_k)\leq C$, we have 
\[
\begin{aligned}
C\geq \cJ_{T_k}(u_k,x_k)&= \int_0^{T_k}\ell_2(t,x_k(t),u_k(t))\,dt+\int_{T_k}^\infty |u_k(t)|^p\,dt\\&\geq \min\{1,\alpha\} \int_0^\infty |u_k(t)|^p \,dt-\int_0^\infty|\gamma(t)|\,dt=\min\{1,\alpha\}\|u_k\|_p^p -\|\gamma\|_1,
\end{aligned}
\]
which implies that the $u_k$ are bounded in $\cU$, and thus, up to a subsequence, there exists $u\in \cU$ such that $u_k\to u$ in $\cU$, and the claim is proven.
By \emph{coercivity w.r.t.\ $x$} 
it follows that also  $x^{u_k,x_0}$ admits a subsequence converging in $\cX$. According to Definition~\ref{defn:EquiCoercivityGeneral}, we have then proved that the sequence   $\cF_{T_k}$ is sequentially equi-coercive and the thesis follows by the Pattern Preservation Theorem \ref{prop:COnvergenceofOptima} . \qed

\

As already explained in Remark \ref{rem_efc}, the condition of \emph{coercivity w.r.t.\ $x$}  stated in Theorem~\ref{MT:StateConstrPatternPReserving2} cannot be ensured by imposing a growth condition like (c1) with respect to the state variable. Nevertheless, such  coercivity  can be forced by prescribing a condition like (c2) on the state, that is a state constraint  $x(\cdot)\in X$ with $X$ compact. Since, as already noted, such state constraint can be incorporated in the cost functional, the compactness hypotheses on $X$ can be seen as a particular case in which the \emph{coercivity w.r.t.\ $x$}
condition   holds. This is shown in  Corollary \ref{MT:StateConstrPatternPReserving2SC} that we are going to prove. 
 Also the specific cases presented in Section~\ref{sec:ParticularCases} satisfy the coercivity condition but, there, the  coercivity w.r.t.\ $x$ is ``intrinsic", i.e., it is satisfied \emph{independently of the cost functional}. In those cases, it is the structure of the state equation itself that provides the required coercivity with respect to the state.

\

\noindent {\em Proof of Corollary \ref{MT:StateConstrPatternPReserving2SC}.} 
 We have only to prove that for any sequence  $0<T_k\to \infty$, the condition of {\em coercivity w.r.t.\ x}   of Theorem \ref{MT:StateConstrPatternPReserving2} is satisfied. 
 
 Let $(u_k,x_k)\in \cU\times \cX$ such that 
$\cF_{T_k}(u_k,x_k)\leq C
$ 
for some $C\geq 0$.  
Let $u_k\to u$ in $\cU$. Since  $x_k=x^{u_k,x_0}$ belongs to the compact state-constraint set $X\subset \R^n$ for any $k\in \N$ (otherwise $\cF_{T_k}(u_k,x_k)$ would be not finite), we have that the sequence $x_k$ is bounded in $L^\infty((0,\infty),\R^n)$. 
Now, using hypothesis~\textnormal{(A0)} we have
\begin{equation}\label{eq:DerivativeProveConvergence}
\left|x'_k(t)\right|\leq |a(t,x_k(t))|+|b(t,x_k(t))||u_k(t)|\leq M(t)+A_{H}(t)D+(N(t)+B_H(t)D)|u_k(t)|
\end{equation}
where $H=X\cup\{0\}$ and  $D=\text{diam}(H)$. 
By using Holder's inequality and the boundedness of $u_k$ in $L^p((0,\infty),\R^m)$, our standing assumptions  imply that $x_k'$ is bounded in $L^q_{\loc}([0,\infty),\R^n)$.
We have thus shown  that the sequence $x_k$ is bounded in $\Wl^{1,q}([0,\infty),\R^n)$, and it thus admits a weakly$^\star$ converging subsequence $x_k\weaks x$ in $\Wl^{1,q}([0,\infty),\R^n)$. Then, by  Lemma~\ref{lem_ucbs}, this implies that $x_k\to x$ in $\cX$. Hence the claimed coercivity. The thesis follows then by applying Theorem  \ref{MT:StateConstrPatternPReserving2}.  \qed



\begin{rem} 
By assumption, the running cost $\ell_2$ satisfies a coercivity assumption with respect to the control variable $u$ (see (c) of Section \ref{sec:MainResults}). Corollary \ref{MT:StateConstrPatternPReserving2SC} shows an example of a class of optimal control problems in which also the additional {\em coercivity w.r.t.\ $x$} (see Theorem \ref{MT:StateConstrPatternPReserving2})   is satisfied. 
Other examples satisfying an even stronger condition are the intrinsically coercive systems (uniformly bounded systems and linear control systems in particular) discussed in  
Section~\ref{sec:ParticularCases}.
\end{rem}




\section{Applications and examples}\label{sec:Examples}
In this section we apply the $\Gamma$-convergence results obtained in the  previous one to specific optimal control problems arising from engineering/biological systems.
\subsection{Optimal switching control problems}
The framework of \emph{switched dynamical systems}
provides a mathematical model for a large class of physical/engineering  phenomena and has been a topic of intense research in the past and recent years; we refer to the monograph~\cite{Lib03,ChitMasSig25} for an overview.
In this framework, given $M$ vector fields $f_1,\dots, f_M:\R^n\to \R^n$, one considers the differential equation
\begin{equation}\label{eq:SwitchIntro}
\dot x(t)=f_{u(t)}(x(t)),
\end{equation}
where $u:(0,\infty)\to \{1,\dots,M\}$ is a so-called \emph{switching signal} selecting, at each instant of time, one subsystem among $f_1,\dots, f_M$ to be ``active''. 

In this subsection we focus on the case of switched systems composed by two \emph{linear subsystems}, and our aim is to characterize the optimal switching policy (with respect to a quadratic cost) in finite and infinite horizon.
More precisely, let us consider $A_1,A_2\in \R^{n\times n}$. For any $x_0\in \R^n$ and any $T\in [0,\infty]$ we want to study the optimal control problem  defined by
\begin{equation}\label{eq:OptimalControlSwitched}
\begin{aligned}
\min_{u\in \cU} &\,J_T(u,x)=\frac{1}{2}\int_0^{T}|x(t)|^2 \,dt,\\
&\begin{cases}
x'(t)=u(t)A_1x(t)+(1-u(t))A_2x(t),\\
x(0)=x_0\in \R^n,
\end{cases}
\end{aligned}
\end{equation}
where the control space is $\cU=L^\infty((0,\infty), [0,1])$, i.e., measurable and essentially bounded functions taking values in the compact and convex set $[0,1]$.  

In the case that $u(t)\in \{0,1\}$ for almost all $t\in[0,T]$, this corresponds to finding the optimal \emph{switching signal} in order to minimize the cost functional $J_T$, i.e., the ``energy'' of the system on the interval $[0,T]$. On the other hand, if $u(t)\in (0,1)$ 
in a non-trivial interval, this corresponds to the fact that the control imposes to follow a \emph{convex combination} of the vector fields defined by $A_1,A_2$, and this might be undesirable in the case of digital/quantized controls, in which only the extreme values $\{0,1\}$  are applicable for physical reasons. With this motivation in mind, in what follows we provide conditions ensuring that the optimal control is bang-bang, i.e., taking values in the discrete set $\{0,1\}$. \\
Defining $B_1=A_2$ and $B_2=A_1-A_2$, and taking again the state space $\cX=W^{1,1}_{\loc}([0,\infty),\R^n)$, we consider the following more  effective rewriting of problem~\eqref{eq:OptimalControlSwitched}  which consists in minimizing  the functional $\mathcal{F}_T:=\cJ_T+\chi_{\cA_{\rm sw}}:\cU\times \cX\to [0,\infty]$  where
\begin{subequations}\label{eq:RelaxedProblem}
\begin{align}
\cJ_T(u,x)&:=\frac{1}{2}\int_0^{T}|x(t)|^2 \,dt\\
\cA_{\rm sw}&:=\{(u,x)\in \cU\times \cX\;:\;x=x^{u,x_0}\}\\
&\ \text{where $x^{u,x_0}$ is solution to the Cauchy problem }\nonumber\\
&\ \begin{cases}
x'(t)=B_1x(t)+u(t)B_2x(t)\\ \label{eq:StateEQ}
x(0)=x_0\in \R^n.
\end{cases}
\end{align}
\end{subequations}
It is easy to see that the state equation~\eqref{eq:StateEQ} satisfies the basic conditions stated in hypothesis~\textnormal{(A0) in Section~\ref{sec:MainResults}}.
Existence of optimal controls for the problem~\eqref{eq:RelaxedProblem} in the case $T<\infty$ follows by classical results, see for example~\cite[Theorem 5.2.1]{BressanPiccoli}.
We want to provide necessary conditions that the optimal controls have to satisfy, making use of the well-known \emph{Pontryagin's minimum principle}, for which  the interested reader is referred to~\cite[Chapter 6]{BressanPiccoli} or~\cite[Chapter 6]{Vinter2010Book}.

To this aim, we introduce the \emph{Hamiltonian function} $H:\R^n\times \R^n\times [0,1]\to \R$, defined  by 
\begin{equation}
H(x,p,u)=\frac{1}{2}x^\top x+p^\top B_1x+up^\top B_2x,
\end{equation}
where we consider $z\in \R^n$ as column vectors, $z^\top$ is the corresponding row vector, and $z^\top w$ is the Euclidean scalar product between $z,w\in \R^n$. This choice for notation agrees with the one usually adopted in this subfield of control theory. 

We now explicitly write the conditions arising from Pontryagin's principle for finite-horizon problems.
\begin{lemma}[Pontryagin principle conditions]\label{lemma:PontryginSwitching}
Let us fix $T\in(0,\infty)$.
Consider any $(\hu,\hx)\in \cU\times \cX$ minimum for $\cF_T$ defined in~\eqref{eq:RelaxedProblem}. Then,
\begin{enumerate}[leftmargin=*]
    \item there exists a \emph{costate} $\hp:[0,T]\to \R^n$ such that $\hp\in W^{1,1}((0,T),\R^n)$, solution to the \emph{adjoint} system 
\begin{equation}\label{eq:AdjointEquation}
\begin{cases}
\hp'(t)=-\hx(t)-B_1^\top\hp(t)-\hu(t)B_2^\top \hp(t),\\
\hp(T)=0;
\end{cases}
\end{equation}
\item given the \emph{switching function} $\varphi:[0,T]\to \R$ defined by $\varphi(t):= \hp^\top(t) B_2\hx(t)$, the control $\hu$ satisfies  \emph{Weierstrass condition}, i.e., for almost all $t\in [0,T]$ it holds that
\begin{equation}\label{eq:Weierstrass}
\hu(t)=\begin{cases}
0\;\;\;&\text{if } \varphi(t)>0,\\
1\;\;\;&\text{if } \varphi(t)<0;
\end{cases}
\end{equation}
\item the Hamiltonian  is constant along $(u,x)$, i.e.,
\begin{equation}\label{eq:ConstancyHamiltonian}
\frac{1}{2}\hx(t)^\top \hx(t)+\hp^\top(t) B_1\hx(t)+\hu(t)\varphi(t)\equiv \frac{1}{2}|\hx(T)|^2\;\;\;\forall \;t\in [0,T].
\end{equation}
\end{enumerate}
\end{lemma}
\begin{proof}
The proof is a direct application of Pontryagin's minimum  principle to our case. For the general statement we refer, for example, to~\cite[Theorem 6.5.1]{BressanPiccoli} or \cite[Theorem 6.2.1]{Vinter2010Book}. 
Since we do not have any constraint on the final position of the state, and the initial state is fixed, it follows that the adjoint state $p$ has final condition equal to $p(T)=0$ (and no initial condition). This, by non-degeneracy, implies that the scalar Lagrange multiplier associated to the cost function (denoted by $\lambda_0$ in~\cite[Theorem 6.5.1]{BressanPiccoli}) can be fixed to be equal to $1$ and it is thus not explicitly reported in the statement. Moreover, since the final time is fixed, we have constancy of the Hamiltonian function, as stated in \emph{(v)} of~\cite[Theorem 6.2.1]{Vinter2010Book}.
\end{proof}

Under an algebraic condition on the matrices $A_1,A_2$, we can ensure that the optimal controls are bang-bang for the finite-horizon problem. We are then able to apply Theorem~\ref{MT:StateConstrPatternPReserving2} proving that the same property holds for an infinite-horizon optimal control.

Given $M\in \R^{n\times n}$, let us denote by  $\cS(M):=\frac{1}{2}(M+M^\top)\in \R^{n\times n}$  its \emph{symmetric part}.

\begin{thm}\label{prop:SwitchedBangBagn}
Suppose that $A_1,A_2\in \R^{n\times n}$ commute, i.e., $[A_1,A_2]:=A_1A_2-A_2A_1=0$ and be such that 
\begin{equation}\label{eq:ConditionSwitchedSystems}
\begin{aligned}
&\mbox{for every }x\in \R^n\setminus\{0\} \text{ such that  } x^\top \cS(A_1-A_2)x=0 \text{ we have }\\
&x^\top \cS(A_1-A_2)A_1x>0 \text{ and } x^\top \cS(A_1-A_2)A_2x>0.
\end{aligned}
\end{equation}
Then,  for any $x_0\in \R^n\setminus\{0\}$ and any $T\in (0,\infty)$, any optimal control of $\cF_T$ in \eqref{eq:RelaxedProblem} is bang-bang of type $1$-$0$ with at most one discontinuity point. 
Alternatively to~\eqref{eq:ConditionSwitchedSystems}, suppose that
\begin{equation}\label{eq:ConditionSwitchedSystems2}
\begin{aligned}
&\mbox{for every }x\in \R^n\setminus\{0\} \text{ such that  } x^\top \cS(A_1-A_2)x=0 \text{ we have}\\
&x^\top \cS(A_1-A_2)A_1x<0 \text{ and } x^\top \cS(A_1-A_2)A_2x<0.
\end{aligned}
\end{equation}
Then,  for any $x_0\in \R^n\setminus\{0\}$ and any $T\in (0,\infty)$, any optimal control of $\cF_T$ in \eqref{eq:RelaxedProblem} is bang-bang of type $0$-$1$ with at most one discontinuity point. 
 Moreover, in both cases, there exists an optimal control $u_\infty$ of $\cF_\infty$ which is bang-bang of the same type with at most one discontinuity point.
\end{thm}
\begin{rem} 
The case $x_0=0$ ruled out in Theorem~\ref{prop:SwitchedBangBagn} is trivial: any control $u\in \cU$ is optimal (for any $T\in [0,+\infty]$), leading to a cost equal to $0$.
\end{rem}
\begin{rem}
Theorem~\ref{prop:SwitchedBangBagn} provides conditions on the subsystem matrices $A_1,A_2\in \R^{n\times n}$ under which the optimal controls are of a bang-bang form with at most one discontinuity point, both for finite and infinite horizon problems. The commutativity hypothesis is classical in control/stabilization of switched systems, see for example~\cite{AgrLib01}. It basically ensures that the reachable set under arbitrary measurable controls is equal to the reachable set under bang-bang controls with at most one switch, see also~\cite[Subsection 2.2.1]{Lib03}.  
The hypothesis in~\eqref{eq:ConditionSwitchedSystems} can be rewritten and checked numerically via the Finsler's Lemma (see for example~\cite[Theorem 2.2]{ZiZongJin10}), which states that~\eqref{eq:ConditionSwitchedSystems} is equivalent to the following proposition
\[
\begin{aligned}
&\exists\;\mu_1 \in \R,\text{ such that }\cS(\cS(A_1-A_2)A_1)+\mu_1\cS(A_1-A_2)\succ 0\;\text{and}\\ &\exists\;\mu_2 \in \R,\text{ such that }\;\cS(\cS(A_1-A_2)A_2)+\mu_2\cS(A_1-A_2)\succ 0,
\end{aligned}
\]
which, in turn, can be checked via semidefinite optimization techniques. Similarly, hypothesis \eqref{eq:ConditionSwitchedSystems2} can be numerically treated using the same tools.
\end{rem}

\noindent{\em Proof of Theorem \ref{prop:SwitchedBangBagn}.}
Let us fix $T\in (0,\infty)$, and consider $(\hu,\hx)\in \cU\times \cX$ any optimal control-state pair of $\cF_T$. Since $x_0\neq 0$, we have that $\hx(t)\neq 0$ for all $t\in [0,T]$, by linearity of the state equation~\eqref{eq:StateEQ}. By Lemma~\ref{lemma:PontryginSwitching}, there exists a costate $\hp:[0,T]\to \R^n$ for which the conditions in the aforementioned lemma are satisfied. We first note  that $[B_2,B_1]=[A_1,A_2]=0$.
Using this equality we have, in the sense of distributions
\begin{eqnarray}\label{eq:VarPhiShort}
\varphi&=&\hp^\top B_2\hx,
\\
\label{eq:derivativePHI}
\varphi'&=&-\hx^\top B_2\hx=-\hx^\top \mathcal{S}(B_2)\hx ,
\\
\varphi''&=& -2\hx^\top \mathcal{S}(B_2)B_1\hx
-2\hu\left (\hx^\top  \mathcal{S}(B_2)B_2\hx\right ).\nonumber
\end{eqnarray}
From this, it can be seen that  $\varphi\in\cC^1([0,T],\R)$ and $\varphi'\in W^{1,1}((0,T),\R)$.
Let us suppose~\eqref{eq:ConditionSwitchedSystems} holds, and
consider the set $\mathcal{T}:=\{ t\in [0,T]\;:\;  \varphi'(t)=0\}$. By the expression~\eqref{eq:VarPhiShort}, we have $\hx(t)^\top \mathcal{S}(A_1-A_2)\hx(t)=0$ for all $t\in \mathcal{T}$. Further computing, we obtain 
\[
\begin{aligned}
\varphi''&= -2\left (\hx^\top \mathcal{S}(B_2)B_1\hx+\hu \hx^\top  \mathcal{S}(B_2)B_2\hx\right )\\&=-2\left (\hx^\top \mathcal{S}(A_1-A_2)A_2\hx+\hu \hx^\top  \mathcal{S}(A_1-A_2)(A_1-A_2)\hx\right )
\\&=-2(1-\hu)\hx^\top \mathcal{S}(A_1-A_2)A_2\hx-2 \hu\hx^\top \mathcal{S}(A_1-A_2)A_1\hx
\end{aligned}
\]
and, thanks to condition~\eqref{eq:ConditionSwitchedSystems}  we have that  $\varphi''<0$  almost everywhere in a neighborhood of $\mathcal{T}$, i.e., $\varphi'$ is locally  strictly decreasing in such neighborhood.

 Then, given $t_0\in \mathcal{T}$, by continuity of $\varphi'$, $t_0$ is actually the \emph{unique} instant for which $\varphi'(t)=0$: otherwise, we will contradict the fact that $\varphi'$ is locally decreasing at every point in which it vanishes. 
More specifically, we have that $\varphi'(t)>0$ for all $t\in [0,t_0)$ and $\varphi'(t)<0$ for all $t_0\in (t_0,T]$ and thus $\varphi$ is strictly increasing in $[0,t_0)$ and strictly decreasing in $(t_0,T]$. Since $ \hp(T)=0$, we have $\varphi(T)=0$ and then there exists at most one $t\in [0,T)$ such that $\varphi(t)=0$.
This, recalling the  Weierstrass condition~\eqref{eq:Weierstrass}, implies that $\hu:(0,\infty)\to \R^m$ has at most one discontinuity point in $[0,T)$, and it is bang-bang of the form $1$-$0$.

The case in which~\eqref{eq:ConditionSwitchedSystems2} holds, leads to $\varphi''<0$ (locally) for all $t\in [0,T]$ such that $\varphi'(t)=0$ and it is thus  similar, leading to bang-bang optimal controls of type $0$\,-$1$.

Now, to treat the infinite-horizon case, we  apply  Theorem~\ref{MT:StateConstrPatternPReserving2} to show that $\cF_T$ is pattern preserving.  

 Let thus verify that all the hypotheses of Theorem~\ref{MT:StateConstrPatternPReserving2} are satisfied.
 Let us  start by checking hypothesis~\textnormal{(A1)} of Section~\ref{sec:MainResults} (we have already observed that hypothesis~\textnormal{(A0)} is satisfied). 
Since $A_1$ and $A_2$ commute, it can be seen via Gr\"onwall's Lemma, that there exists  $M\geq 1$ such that, for every $\varepsilon>0$, it holds that 
\begin{equation}\label{eq:boundednessSwitching}
|x^{u,x_0}(t)|\leq M e^{(\lambda+\varepsilon) t} |x_0| \;\;\forall x_0\in \R^n,\forall u\in  L^\infty((0,\infty),[0,1]),\forall t\ge0,
\end{equation}
with $\lambda=\max_{i\in\{1,2\}}\lambda(A_i)$ and where, given $A\in \R^{n\times n}$, $\lambda(A)$ denotes the~\emph{spectral abscissa of $A$}, defined as $\displaystyle \lambda(A) : =\max_{j\in\{1,\dots, n\}} \text{Re}(\lambda_j)$, where $\lambda_1,\dots,\lambda_n$ denote the (possibly complex) eigenvalues of $A$, see~\cite[Section~2.2]{Lib03}. 
Inequality~\eqref{eq:boundednessSwitching} implies that hypothesis~\textnormal{(A1)} holds: 
indeed, for any $u\in L^\infty((0,\infty),[0,1])$ we have that $x^{u,x_0}$ is bounded in $L_{\loc}^\infty([0,\infty),\R^n)$ and is thus globally defined. \\
The hypotheses \emph{(a)} and \emph{(b)} and \emph{(c)} of Theorem~\ref{MT:StateConstrPatternPReserving2} are trivially satisfied by defining $\ell_1(t,x)=\frac{|x|^2}{2}$, $\ell_2(t,x,u)=0$ and $u_g\equiv 0$, and recalling that $U=[0,1]$ is compact.
 
Among the hypotheses of Theorem~\ref{MT:StateConstrPatternPReserving2}, it remains only to prove the \emph{coercivity w.r.t.\ $x$}. Let us then consider  sequences $0<T_k\to\infty$  and  $(u_k,x_k)$  such that $
\mathcal{F}_{T_k}(u_k,x_k)\leq C$ for all $k\in \N$ and a prescribed $C\geq 0$. By definition of $\mathcal{F}_{T_k}(u_k,x_k)$, we have that $u_k\in  L^\infty((0,\infty),[0,1])$ and $x_k=x^{u_k,x_0}$.
Again by~\eqref{eq:boundednessSwitching}, 
the sequence $x_k=x^{u_k,x_0}$ is bounded in $L_{\loc}^\infty([0,\infty),\R^n)$ and, since $\frac{d}{dt}x^{u_k,x_0}(t)=B_1x^{u_k,x_0}(t)+u_k(t)B_2x^{u_k,x_0}(t)$, then it is bounded also in $\Wl^{1,\infty}([0,\infty),\R^n)$. Since it is bounded, it admits a converging subsequence $x_{k_\ell}\to x$ in the chosen topology, i.e., the strong topology of $L^\infty_{loc}([0, \infty), \R^n)$.

We can thus apply Theorem~\ref{MT:StateConstrPatternPReserving2}, obtaining that the considered family $\cF_T$ is \emph{pattern preserving} with respect to every sequence $0< T_k\to\infty$.

Let us now consider a sequence $0<T_k\to \infty$ and  an optimal control $u_k\in \cU$ of $\cF_{T_k}$. We have already proved that $u_k$ has at most a discontinuity point in $[0,T_k)$ and it is always of the same bang-bang type, $0$\,-$1$ or $1$-\,$0$, for every $k\in \N$. Let us suppose, without loss of generality,  that we are in the $0$\,-$1$ case.
Let us denote  by $\tau^k$ the discontinuity point of $u_k$ in $(0,T_k)$, if any; otherwise, if $u_k$ is constant equal to $0$, we set $\tau^k=T_k$ and if $u_k$ is constant equal to $1$ we set $\tau^k=0$.  There exists a subsequence of $T_k$ (without relabeling) such that $\lim_{k\to \infty}\tau^{k}$ exists in $(0,\infty]$; let us denote such limiting point by $\tau^\infty$. 
 Since the family $\cF_T$ has been proven to be pattern preserving, via Remark~\ref{rem:BangsBangs}, we have
 that $u_\infty:(0,\infty)\to [0,1]$ defined by
\[
u_\infty(t)=\begin{cases}0 \;\;\;\text{if } t\in (0,\tau^\infty),\\
1 \;\;\;\text{if } t\in [\tau^\infty, \infty),
\end{cases}
\]
is an optimal control for $\cF_\infty$, concluding the proof.  \qed

\subsection{Optimal control of SIR models }
Optimal control problems for epidemics driven by 
Susceptible-Infectious-Recovered (SIR) models introduced by Kermack and McKendrick (\cite{kermack1927contribution})
has been intensively studied in the past decades. For an overview, we refer to~\cite{anderson1992infectious,Behncke,hansen2011optimal,SFT2022} and references therein.  
We briefly recall here this framework and we then apply the results previously developed.\\
We consider  a two-dimensional SIR model defined by
\begin{equation}\label{eq:SIR2}
\begin{cases}
s'(t)=-b(t)\, s(t)\, i(t)-v(t)s(t),%
\\[1ex]
i'(t)= b(t)\, s(t)\, i(t)-\gamma i(t),\\[1ex]
s(0)=s_0,\ i(0)=i_0,
\end{cases}
\end{equation}
 where the states $s$ and $i$ are the population densities of susceptible and infectious individuals, respectively, while 
 the measurable parameters $b:(0,\infty)\to [\beta_\star, \beta]$ and $v:(0,\infty)\to [0,v_M]$ represent the transmission and vaccination rate (respectively) and can be considered as external or controlled parameters. The constant coefficients $\beta_\star,\beta\in (0,1)$ and $v_M,\gamma>0$ 
are, respectively, the minimal and maximal transmission rates, the maximal vaccination rate and the (constant) recovery rate. 
Whenever measurable inputs $v:(0,\infty)\to [0,v_M]$, $b:(0,\infty)\to [\beta_\star,\beta]$ and initial condition $x_0=(s(0),i(0)) =(s_0,i_0)\in \R^2$ are fixed, the unique solution to \eqref{eq:SIR2} will be denoted by $x^{b,v,x_0}$.

Let us state the following preliminary invariance result.
\begin{lemma}\label{lemma:TriangleInvariant} 
For any choice of parameters  $\beta_\star,\beta\in (0,1)$ and $v_M,\gamma>0$, the 
unit triangle $$\T=\{(s,i)\in \R^2\;\vert\; s\geq 0,\;i\geq 0,\;s+i\leq 1\}.
$$
is forward invariant for~\eqref{eq:SIR2} with respect to $[\beta_\star,\beta]\times [0,v_M]$, i.e.,
\[
x_0\in \T\;\;\Rightarrow\;\;x^{b,v,x_0}(t)\in \T,\;\;\forall \,b\in L^\infty((0,\infty),[\beta_\star,\beta]),\;\;\forall v\in L^\infty((0,\infty),[0,v_M]),\;\forall \,t\ge0.
\]
\end{lemma}

\begin{proof}
Since the right-hand side in~\eqref{eq:SIR2} satisfies hypothesis~\textnormal{(A0)}  we have local existence and uniqueness of an absolutely continuous solution, for every $u\in \cU$ and $(s_0,i_0)\in\T$.

Let $[0,c)$,  with $c>0$, be a time interval in which the unique solution $(s,i)$ exists.  
Let us suppose that $(s_0,i_0)\in \T$ and
 claim that $s(t)\ge0$ and $i(t)\ge0$ for every $t\in [0,c)$. Indeed,
considering $i$ as a coefficient, the function $s$ is the unique solution to the linear Cauchy problem
$$
\begin{cases}
s'=-\big(b(t) i(t)+v(t)\big) s, \\
s(0)=s_0,
\end{cases}
$$
which is given by
$$
s(t)=s_0 e^{-\int_0^t[b(\xi) i(\xi)+v(\xi)]\,d\xi}\quad\forall\,t\in [0,c).
$$
Hence, $s$ is non-negative in $[0,c)$. Similarly, one can prove that also $i$ is non-negative. 

Then, by summing the two equations of the system we get
\begin{equation}\label{d(s+i)}
(s+i)'=-vs-\gamma i\le 0\ \mbox{ in }[0,c),
\end{equation}
which implies that $s+i$ is non-increasing in $[0,c)$. Then we have
\begin{equation}\label{s+ile}
0\le s(t)+i(t)\le s_0+i_0\le1\quad\forall\,t\in [0,c).
\end{equation}
Hence, the solutions are  bounded and remain within the compact triangle $\T$. This boundedness implies that they exists for every $t\in[0,\infty)$ and all properties above hold with $c=\infty$. 
\end{proof}

\noindent Simply looking at the equations, one also sees that $i$ and
$s$ have then bounded derivatives.   This, in particular,  implies that $x^{b,v,x_0}\in W^{1,\infty}((0,\infty),\R^2)$ for any $b\in L^\infty((0,\infty),[\beta_\star,\beta])$ and  $v\in L^\infty((0,\infty),[0,v_M])$.

Given $x_0\in \T$, we study  optimal control problems for system \eqref{eq:SIR2} which consists in minimizing, over all controls $b$ or $v$ and the corresponding epidemic trajectories $s$ and $i$, the objective functional
\begin{equation}\label{eq:CostSIR}
J_T(b,v,s,i)=\int_0^{T}\lambda_b(\beta-b(t))+ \lambda_v v(t)+\lambda_i i(t) \,dt,
\end{equation}
under a state constraint
\begin{equation}\label{eq:ConstraintSIR}
i(t)\le i_M\quad\forall\,t\in[0,T],
\end{equation}
where $i_M>0$ represents a safety threshold for the intensive care units capacity. The constant coefficients $\lambda_b,\lambda_v\ge0$ and  $\lambda_i\ge0$ modulate the economic and health-related costs of infection with respect to the cost of vaccination/reduction of the contact rate. The results now illustrated, in a condensed form, are borrowed from~\cite{FreddiGoreac23,DellaRossaFreddi24}. We analyze separately the case with no control on the contact rate (i.e., $b\equiv \beta$) and the case with no control on the vaccination rate (i.e., $v\equiv 0$).

\def\vac{{\rm vac}}

Let us start from the case in which $b\equiv \beta$ and the unique control variable is the vaccination rate $v$. Since a state-constraint is involved, whenever $i_M<1$, the maximal viable set $B$, that is, the set of initial conditions such that the state-constraint is satisfied along the whole epidemic horizon, is a proper subset of the invariant triangle $\T$. This set $B$ has been characterized in \cite[Theorem 3.5(2)]{DellaRossaFreddi24}. We refer to that paper for a detailed description. 

\begin{prop}[optimal vaccination]
Let $T\in [0,\infty]$ and $x_0=(s_0,i_0)\in B$. 
Let us consider the optimal control problem $\cF_T$ composed by
\begin{enumerate}[leftmargin=*,label=(\alph*)]
\item  the SIR state-equation~\eqref{eq:SIR2} with $b\equiv \beta>0$,
    \item the cost functional $J_T$ given in~\eqref{eq:CostSIR}, to be minimized over the control space $\cU=L^\infty((0,\infty),U)$ with $U=[0,v_M]$ and state space $\cX=W^{1,1}_\loc([0,\infty),\R^2)$,
    \item the state-constraint~\eqref{eq:ConstraintSIR}. 
\end{enumerate}  
Then, for any finite $T>0$ large enough, any optimal control $v_T:(0,\infty)\to [0,v_M]$ of $\cF_T$ is of the bang-bang form
\[
v_T(t)=\begin{cases}
v_M&\text{if }t\in (0, \tau_1^T),\\
0&\text{if }t\in (\tau_1^T,\infty),
\end{cases}
\]
for some $\tau_1^T\in[0,T]$.
The  infinite-horizon problem ($T=\infty$) admits an optimal control of the same bang-bang structure, for a $\tau_1^\infty<\infty$.  
\end{prop}
The case with finite $T>0$ is obtained by applying the Pontryagin's principle to the problem under consideration, for the details we refer to~\cite[Theorem 4.6]{DellaRossaFreddi24}. It is now easy to check that all the  hypotheses of Theorem~\ref{MT:StateConstrPatternPReserving2} are satisfied.
The only nontrivial task is to verify \emph{coercivity w.r.t.\ $x$} of the joint functionals. This can be verified via Proposition~\ref{Lemma:LocallyLispchitzSol} since solutions are confined in the  compact set $\T$, as stated in Lemma~\ref{lemma:TriangleInvariant}.  Moreover, the finiteness of $\tau_1^\infty$  can be proven, since there exists a finite upper bound for $t_1^T$ uniform on $T$.
For such details, a more direct proof and a discussion about uniqueness of the optimal control we refer to~\cite[Section 5]{DellaRossaFreddi24}.

\medskip

We now recall a similar result in the case of \emph{non-pharmaceutical} control, i.e.,  when no vaccination is possible, and the only control parameter is given by the function $b:(0,\infty)\to [\beta,\beta^\star]$. Also in this case, whenever $i_M<1$,  the maximal viable set $\cB$ is a proper subset of $\mathbb{T}$ and has been characterized in \cite[Item 5.\ of Theorem 2.3]{AvrFre22}, see also \cite{AFG22cor}, to which we refer also for further details. Moreover, since this is just a matter of examples and for the sake of simplicity, we restrict ourselves to the simpler case in which the cost depends only the control (i.e., $\lambda_i=0$). Nevertheless, in \cite{FreddiGoreac23} also the general case has been  handled.

{
\begin{prop}[Optimal non-pharmaceutical intervention]
Given $x_0=(s_0,i_0)\in \cB$ and $T\in [0,\infty]$, 
let us consider the optimal control problem $\cF_T$ composed by
\begin{enumerate}[leftmargin=*,label=(\alph*)]
\item  the SIR state-equation~\eqref{eq:SIR2} with $v\equiv 0$,
    \item the cost functional $J_T$ given in~\eqref{eq:CostSIR} with $\lambda_i=0$, to be minimized over the control space $\cU=L^\infty((0,\infty),U)$ with  $U=[\beta_\star,\beta]$  and state space $\cX=W^{1,1}_\loc([0,\infty),\R^2)$,
    \item the state-constraint~\eqref{eq:ConstraintSIR}. 
\end{enumerate}   Then, for any finite $T>0$ large enough, any optimal control $b_T:(0,\infty)\to [\beta_\star,\beta]$ of $\cF_T$ is of the  form
\[
 b_T(t)=\begin{cases}
\beta&\mbox{ if }t\in(0,\tau_1^T),\\
\beta_\star&\mbox{ if }t\in(\tau_1^T,\tau_2^T),\\
\dis\beta-\frac{\gamma}{s(\tau_2^T)+\gamma i_M(\tau_2^T-t)}&\mbox{ if }t\in(\tau_2^T,\tau_3^T),\\[2ex]
\beta&\mbox{ if }t\in(\tau_3^T,\infty), 
\end{cases}
\]
for some $0\leq \tau_1^T\leq\tau_2^T\leq\tau_3^T\le T$.
The  infinite-horizon problem (\,$T=\infty$) admits an optimal control of the same structure, with $\tau_3^\infty<\infty$. 
\end{prop}

Again, the finite-horizon case $T>0$ is obtained by applying the Pontryagin's principle (see~\cite{FreddiGoreac23}). Since, for initial conditions $x_0\in \cB$,  all the  hypotheses of Theorem~\ref{MT:StateConstrPatternPReserving2} are satisfied, the infinite-horizon case directly follows. The finiteness of $\tau_3^\infty$ is proven in~\cite[Theorem 9.9]{FreddiGoreac23} to which we refer for the details and for an alternative and more ad-hoc proof.
}

\section{Conclusions}
In this work, we have introduced a general framework to study the preservation of optimal control patterns as the time horizon extends to infinity. By employing the notion of $\Gamma$-convergence as crucial technical tool, we have provided sufficient conditions ensuring that the structural properties of optimal controls  persist in the infinite-horizon case. Our  approach naturally incorporates the possibility of the presence of  state constraints and provides a rigorous variational foundation for analyzing long-time behavior in optimal control problems. 
 The theoretical results have been eventually  illustrated through examples, with a specific focus on switched systems and epidemic control. These results suggest promising directions for further research, including the extension to broader classes of control problems and/or more general variational problems.
\appendix
\section{Local Sobolev function spaces}\label{subsec_locSob}
In this section 
we introduce a notion of local-in-time Sobolev spaces, suitable to our analysis. In doing it, we assume as known the basic concepts of Sobolev spaces $W^{m,p}$ on a bounded interval and their duals, as well as the notion of strong (or norm), weak and weak$^{\star}$ topologies on such classical spaces. The  unaccustomed reader is referred to \cite{Adams,Brezis2011}.      

\begin{defn}
Given $p\in [1,\infty]$,  $m\in\N$, and $\tau\in(0,\infty]$,  we define the set of \emph{locally $(m,p)$-Sobolev functions on $[0,\tau)$} by
\[
\Wl^{m,p}([0,\tau),\R^n)=\left \{x\in\cD'((0,\tau),\R^n)\; :\;  x_{|(0,T)}\in W^{m,p}((0,T),\R^n)\ \forall\, T\in(0,\tau)\right\},
\]
\noindent where $x_{(0,T)}$ denotes the restriction of the distribution $x$ to all test functions with compact support in $(0,T)$. 
\end{defn}

When $m=0$ we use also the notation $L_\loc^p([0,\tau),\R^n):=\Wl^{0,p}([0,\tau),\R^n)$. 

Since the intervals $(0,T)$ are bounded (i.e., $T<\infty$), then we have
\begin{equation}\label{eq_inclWloc}
1\le p\le q\ \implies\ \Wl^{m,q}([0,\tau),\R^n)\subseteq \Wl^{m,p}([0,\tau),\R^n).
\end{equation}
Moreover, it is easy to check that $L^p_\loc([0,\tau))$ is contained in the classically defined space $L^p_\loc(0,\tau)$ on the open set $(0,\tau)$, and the inclusion is strict (e.g., $x(t)=1/t\in L^p_\loc(0,\tau)\setminus L^p_\loc([0,\tau))$\,). 
Consequently, for any $m\in \N$, we have 
$$
\Wl^{m,p}([0,\tau),\R^n)\subset \Wl^{m,p}((0,\tau),\R^n)\subset\cD'((0,\tau);\R^n).
$$
On $\Wl^{m,p}([0,\tau),\R^n)$, we can naturally consider the family of separating \emph{seminorms} $\|\cdot\|_T$ parametrized by $T\in(0,\tau)$, given by
\[
\|x\|_T:=\|x_{\vert(0,T)}\|_{W^{m,p}(0,T)}
\]
where $\|\cdot\|_{W^{m,p}(0,T)}$ denotes the usual norm of the Sobolev space $W^{m,p}((0,T),\R^n)$.   The following definition is standard in this functional framework.

\begin{defn}\label{def_bddlctvs}
A set $\cH\subset \Wl^{m,p}([0,\tau),\R^n)$ is said to be \emph{bounded} if it is bounded in $W^{m,p}((0,T),\R^n)$ for every $T\in(0,\tau)$, that is, for every such $T$ there exists  $C_T\geq 0$ such that
\[
\|x\|_{T}\leq C_T\quad \forall x\in \cH.
\]
\end{defn}

\noindent  The topology induced by such a family of seminorms  (called {\em strong topology}) makes  $\Wl^{m,p}([0,\tau),\R^n)$ a locally convex topological vector space (see \cite[Theorem 8.1]{Bianchini}). A basis $\cB$ of neighborhoods of $0$ in this topology is defined by
$$
\cB=\Big\{\bigcap_{i=1}^r\big\{x\; :\; \|x\|_{T_i}<\frac{1}{h}\big\}\; :\; T_i\in(0,\tau),\ i=1,...,r,\mbox{ and } r,h\in\N\Big\}.
$$
Every increasing sequence of positive real numbers $T_k\to\tau^-$ (a so-called {\em defining sequence}) determines  a sequence of seminorms that defines  $W^{m,p}_\loc([0,\tau);\R^n)$ as an {\em $LF$-space} or, equivalently,  {\em countable strict inductive limit} of Frech\'et spaces 
(see, for instance, \cite[Chapter 13]{Treves}). The topology generated by the corresponding countable family of seminorms is the same generated by the whole family, and turns out to be  described by the metric
\[
d(x,y):=\sup_{k\in \N}\ 2^{-k}\frac{\|x-y\|_{T_k}}{1+\|x-y\|_{T_k}}.
\]
With this metric, the space $\Wl^{m,p}([0,\tau),\R^n)$ is complete. 

\begin{prop}\label{prop_stchmp}
Let $x,x_k\in\Wl^{m,p}([0,\tau),\R^n)$ for every $k\in\N$. 
The following propositions are equivalent.
\begin{enumerate}
\item[1.] $x_k\to x$ strongly in $\Wl^{m,p}([0,\tau),\R^n)$.
\item[2.] $x_k\to x$ strongly in $W^{m,p}((0,T),\R^n)$ for every $T\in(0,\tau)$.
\end{enumerate}
\end{prop}

\begin{proof}
First of all, let us note that it suffices to prove the equivalence in the case $x=0$. In this case,  
{\em1.}\ is equivalent to say that for every $B\in\cB$ there exists $k_B\in\N$ such that $x_k\in B$ for every $k>k_B$. 
To prove that {\em1.}\ implies {\em2.}\  we fix $T\in(0,\tau)$ and $\varepsilon=\frac{1}{h}$.  
By taking $B=\{z\ :\ \|z\|_T<\frac{1}{h}\}$ we have that there exists $k_h\in\N$ such that $x_k\in B$ for every $k>k_h$. Since this holds for every $h\in\N\setminus\{0\}$ then we have proved that $x_k\to x$ strongly in $W^{m,p}((0,T),\R^n)$ as claimed in {\em2.}\\
Conversely, given any $B\in\cB$, there exists a finite set of times $T_i$, $i=1,...,r$ and $h\in\N$ such that
$$
B=\bigcap_{i=1}^r\big\{z\; :\; \|z\|_{T_i}<\frac{1}{h}\big\}.
$$
By {\em2.}, there exists $k_{h,i}\in\N$ such that 
$$
\|x_k\|_{T_i}<\frac{1}{h}\quad \forall\,i=1,...r,\;\forall k>k_{h,i},
$$
which means that $x_k\in B$ for every $k>k_B:=\max\{k_{h,i}\;\vert\;i=1,\dots, r\,\}$, and thus proves {\em1}.
\end{proof}

\begin{rem}
Of course, condition {\em2.} of Proposition \ref{prop_stchmp}, as well as the boundedness condition of Definition \ref{def_bddlctvs}, needs to be checked just on  an increasing sequence of positive real numbers $T_k\to\tau^-$.  This provides an alternative  way to see that every defining sequence generates the same strongly converging sequences. 
\end{rem}

Let us now restrict our analysis to the case in which $p\in(1,\infty]$. 
We endow the local Sobolev space $\Wl^{m,p}([0,\tau),\R^n)$ with the following notion of weak convergence.

\begin{definition} Let $p\in(1,\infty]$.  Let $(x_k)$ be a sequence in $\Wl^{m,p}([0,\tau),\R^n)$ and $x\in\cD'((0,\tau),\R^n)$. 
We say that 
$(x_k)$ {\em weakly$^\star$ converges to $x$ in $\Wl^{m,p}([0,\tau),\R^n)$}, and we write $x_k\weaks x$, if 
\begin{enumerate}
\item[1.]  $(x_k)$ is bounded in  
$\Wl^{m,p}([0,\tau),\R^n)$, 
\item[2.] $x_k\wto x$ in the weak sense of distributions of $\cD'((0,\tau),\R^n)$.
\end{enumerate}
\end{definition}
 
\noindent 
The choice to use the term ``weak$^\star$'' for every $p\in(1,\infty]$ 
is consistent with the notation adopted in rest of the paper, whenever the non-reflexive case $p=1$ is excluded.      

It is easy to see that also the weak$^\star$ convergence can be characterized in the same way as we have done before for the strong one.

\begin{prop}\label{prop_wchmp}  Let $p\in(1,\infty]$. 
Let $x,x_k\in\Wl^{m,p}([0,\tau),\R^n)$ for every $k\in\N$.
The following propositions are equivalent.
\begin{enumerate}
\item[1.] $x_k\weaks x$   in $\Wl^{m,p}([0,\tau),\R^n)$.
\item[2.] $x_k\weaks x$ in $W^{m,p}((0,T),\R^n)$ for every $T\in(0,\tau)$.
\end{enumerate}
\end{prop}

\begin{proof}
Let us prove that {\em1.}$\implies${\em2.}
Let $T\in(0,\tau)$. By definition of weak$^\star$ convergence we have that $x_k\wto x$ in $\cD'((0,T),\R^n)$ and $(x_k)$ is bounded in $W^{m,p}((0,T),\R^n)$. This implies that the sequence converges weakly$^\star$ to $x$ also in  $W^{m,p}((0,T),\R^n)$. 

Let us prove that {\em2.}$\implies${\em1.} Being weakly$^\star$ converging in $W^{m,p}((0,T),\R^m)$ for every $T\in(0,\tau)$ the sequence $(x_k)$ is bounded in the norm $\|\cdot\|_T$. Hence, it is bounded in $\Wl^{m,p}([0,\tau),\R^n)$. Moreover, for every $\varphi\in \cD'((0,\tau),\R^n)$ there exists  $T\in(0,\tau)$ such that ${\rm supp}(\varphi)\subset(0,T)$ and, hence, 
$\langle x_k,\varphi\rangle\to\langle x,\varphi\rangle$. Thus $x_k\wto x$ in $\cD'((0,\tau),\R^n)$ and, by boundedness, we can conclude that {\em1.}\ holds. 
\end{proof}

\begin{rem} By the definitions above, immediately follows that the inclusion \eqref{eq_inclWloc} is continuous when the spaces are  both endowed with the strong (respectively, weak$^\star$) topology. 
\end{rem}

The following sequential compactness theorem holds.

\begin{theorem}[sequential compactness]\label{th_comp_wloc}  Let $p\in(1,\infty]$. 
Let $(x_k)$ be a bounded sequence in $\Wl^{m,p}([0,\tau),\R^n)$. Then, there exists $x\in\Wl^{m,p}([0,\tau),\R^n)$ and a subsequence $(x_{k_h})$ of $(x_k)$ that weakly$^\star$ converges to $x$  in $\Wl^{m,p}([0,\tau),\R^n)$.  
\end{theorem}

\begin{proof}
By definition,  $(x_k)$ is bounded in $W^{m,p}((0,T),\R^n)$ for every $T\in(0,\tau)$. Let us consider an increasing sequence of positive real numbers $T_j\to\tau^-$.

For $T=T_1$, by Alaoglu's Compactness Theorem, there exists a subsequence $(x_{k_h^1})$ of $(x_k)$ and $x^1\in W^{m,p}((0,T_1),\R^n)$ such that $x_{k_h^1}\weaks x^1$ in $W^{m,p}((0,T_1),\R^n)$. 

For $T=T_2$, there exists a subsequence $(x_{k_h^2})$ of $(x_{k_h^1})$ and $x^2\in W^{m,p}((0,T_2),\R^n)$ such that $x_{k_h^2}\weaks x^2$ in $W^{m,p}((0,T_2),\R^n)$.  Moreover, since the latter weakly$^\star$ converges also in $W^{m,p}((0,T_1),\R^n)$, by uniqueness of the limit we have $x^2_{|(0,T_1)}=x^1$.

By induction, for every $T=T_j\in\N$ ($j\ge2$) there exists a subsequence $(x_{k_h^j})_h$ of $(x_{k_h^{j-1}})_h$ and $x^j\in W^{m,p}((0,T_j),\R^n)$ such that $x_{k_h^j}\weaks x^j$ in $W^{m,p}((0,T_j),\R^n)$ as $h\to\infty$.  Moreover,  by construction
\begin{equation}\label{eq_xh-1}
x^j_{|(0,T_{j-1})}=x^{j-1}.
\end{equation}
Let us now define the map 
$x:\cD((0,\tau),\R^n)\to\R$ by $\langle x,\varphi\rangle:=\langle x^{j_\varphi},\varphi\rangle$ where $j_\varphi$ is any natural such that ${\rm supp}\,\varphi\subset(0,T_{j_\varphi})$. By \eqref{eq_xh-1}, the map $x$ is well defined. Moreover, $x\in\cD'((0,\tau),\R^n)$ and the sequence
$x_h:=x_{k_h^h}$ is a subsequence of $(x_k)$ that weakly$^\star$ converges to $x$. Indeed,  for every $\varphi\in\cD((0,\tau),\R^n)$ there exists $j\in\N$ such that 
${\rm supp}\,\varphi\subset(0,T_j)$. By definition of the map $x$, and 
since $x_{k_h^h}$ is a subsequence of $x_{k_h^j}$ for $h\ge j$, then we have  
$\langle x_h,\varphi\rangle=\langle x_{k_h^h},\varphi\rangle\to \langle x^j,\varphi\rangle=\langle x,\varphi\rangle$. Thus, $x_h\wto x$ in $\cD'$. Since this subsequence is also (trivially) bounded in $\Wl^{m,p}([0,\tau),\R^n)$, then it converges weakly$^\star$ in the latter space and the theorem is proven.
\end{proof}


The following theorem states that the weak$^\star$ convergence in the space $\Wl^{m,p}([0,\tau),\R^n)$ 
satisfies the Urysohn property of convergence structures.

\begin{theorem}[Urysohn    property]\label{th_Ury_Wloc}
 Let $p\in(1,\infty]$. Let $x\in\Wl^{m,p}([0,\tau),\R^n)$ and $(x_k)$ be a bounded sequence in the same space. If every subsequence of $(x_k)$ admits a further subsequence weakly$^\star$  converging to $x$, then the whole sequence $x_k$ weakly$^\star$ converges to $x$.   
\end{theorem}

\begin{proof}
By contradiction, suppose that $x_k\not\wto x$ in $\cD'((0,\tau),\R^n)$. Then,  there exists $\varphi\in\cD'((0,\tau),\R^n)$ and $\varepsilon>0$ such that 
$$
\forall\,h\in\N\ \exists\,k_h>h\; :\; |
\langle x_{k_h},\varphi\rangle- \langle x,\varphi\rangle|\ge\varepsilon.
$$
Then, the sequence $(x_{k_h})$ does not admit any  weakly$^\star$ converging subsequence, against the assumption of the theorem.  
\end{proof}

Let us remark that (as it is  well known) on the bounded interval  $(0,T)$ the inclusion
$$
W^{1,1}(0,T)\subset L^\infty(0,T)
$$
is continuous but not compact. As a consequence, 
the inclusion
$$
W^{1,1}_\loc([0,\infty);\R^n)\subset L^\infty_\loc([0,\infty);\R^n)
$$
inherits the same properties. 
Moreover, since $L^\infty_\loc([0,\infty);\R^n)$ is an LF-space, the
topology of strong convergence  is metrizable. Hence, the space $W^{1,1}_\loc([0,\infty);\R^n)$,  as a topological subspace of $L^\infty_\loc([0,\infty);\R^n)$ with the strong topology, is a metric space.


On the contrary, when the summability exponent $p$ is strictly larger than $1$, the immersion turns out to be compact, as the following lemma states.

\begin{lemma}\label{lem_ucbs}
 Let $p\in(1,\infty]$.  The space $\Wl^{1,p}([0,\tau),\R^n)$ is compactly embedded in $L^\infty_\loc([0,\tau),\R^n)$. 
In particular, weakly$^\star$ converging sequences in $\Wl^{1,p}([0,\tau),\R^n)$  are strongly converging in $L^\infty_\loc([0,\tau),\R^n)$, that is, uniformly converging on the compact subsets of $[0,\tau)$. 
\end{lemma}

\begin{proof}
The result follows by the well-known compactness of the immersion of $W^{1,p}((0,T),\R^n)$ in $L^\infty((0,T),\R^n)$ for $p\in(1,\infty]$,   and by the characterization of the strong and weak$^\star$ convergences given before. 
\end{proof}

\noindent
{\bf Acknowledgements.}  
 The authors are members of GNAMPA--INdAM.

\

\bibliographystyle{plain}
\bibliography{bibnol}

\end{document}